\documentclass[a4paper,11pt,reqno]{amsart}
\usepackage{graphicx}
\usepackage{a4wide}
\usepackage{amssymb}
\usepackage{amsthm}
\usepackage{eucal}
\usepackage[pdftex,colorlinks]{hyperref}

\numberwithin{equation}{section}

\newcommand{\R}{\mathbb R}
\newcommand{\C}{\mathbb C}
\newcommand{\N}{\mathbb N}
\newcommand{\Z}{\mathbb Z}

\newcommand{\T}{\mathbb T}

\newcommand{\cI}{\mathbb I}
\newcommand{\ii}{\int\!\!\!\int }

\newcommand{\be}{\begin{equation}}
\newcommand{\ee}{\end{equation}}
\newcommand{\ba}{\begin{eqnarray}}
\newcommand{\ea}{\end{eqnarray}}

\def\cf{\rightharpoonup}
\newtheorem{theorem}{Theorem}[section]
\newtheorem{proposition}[theorem]{Proposition}
\newtheorem{remark}[theorem]{Remark}
\newtheorem{lemma}[theorem]{Lemma}
\newtheorem{corollary}[theorem]{Corollary}

\begin{document}

\title[UCP and control for BBM]{Unique continuation property and control for the Benjamin-Bona-Mahony equation on the torus}

\author{Lionel Rosier}
\address{Institut Elie Cartan, UMR 7502 UHP/CNRS/INRIA,
B.P. 70239, 54506 Vand\oe uvre-l\`es-Nancy Cedex, France}
\email{rosier@iecn.u-nancy.fr}

\author{Bing-Yu Zhang}
\address{Department  of Mathematical Sciences, University of Cincinnati, Cincinnati, Ohio 45221, USA}
\email{bzhang@math.uc.edu}

\keywords{Unique Continuation Property; Benjamin-Bona-Mahony equation; Korteweg-de Vries equation;
Moving point control; Exact controllability; Stabilization}

\subjclass{}

\begin{abstract}
We consider the Benjamin-Bona-Mahony (BBM) equation on the one
dimensional torus $\T =\R/(2\pi \Z )$. We prove a Unique
Continuation Property (UCP) for small data in $H^1(\T )$ with
nonnegative zero means. Next we extend the UCP to certain BBM-like
equations, including the equal width wave equation and the KdV-BBM
equation. Applications to the stabilization of the above equations
are given. In particular, we show that when  an internal  control
acting on a moving interval is applied in BBM equation,  then a
semiglobal exponential stabilization can be derived in $H^s(\T )$ for any $s\ge 1$. Furthermore, we
prove that the BBM equation with a moving control is also locally
exactly controllable in $H^s(\T)$ for any $s \ge 0$ and
globally exactly controllable in $H^s(\T)$ for any $s\geq 1$.
\end{abstract}

\maketitle
\section{Introduction}
We are concerned here with the Benjamin-Bona-Mahony  (BBM) equation
\be
u_t-u_{txx} + u_x + uu_x=0
\label{A1}
\ee
that was proposed in \cite{BBM} as an alternative to the Korteweg-de Vries (KdV) equation
\be
u_t+u_{xxx} + u_x + uu_x=0
\label{A2}
\ee
as a model for the propagation of one-dimensional, unidirectional small amplitude
long waves in nonlinear dispersive media. In the context of shallow-water waves, $u=u(x,t)$ represents the displacement of the water surface at location $x$ and time $t$.
In this paper, we shall assume that $x\in \R$ or $x\in \T=\R/(2\pi\Z)$ (the one-dimensional torus).
\eqref{A1} is often obtained from \eqref{A2} in the derivation
of the surface equation by noticing that, in the considered regime, $u_x\sim -u_t$, so that $u_{xxx}\sim -u_{txx}$.  The dispersive term $-u_{txx}$ has a strong smoothing effect,
thanks to which the wellposedness theory of \eqref{A1} is dramatically easier than for \eqref{A2} (see \cite{BBM,BT,Roumegoux} and the references therein).
Numerics often involve the BBM equation, or the KdV-BBM equation (see below), because of the regularization provided by the term $-u_{txx}$.
On the other hand, \eqref{A1} is not integrable and it has only three invariants of motion \cite{HESG,Olver}.

In this paper, we investigate the Unique Continuation Property (UCP) of BBM and its applications to the Control Theory for  \eqref{A1}.
We say that the UCP holds in some class $X$ of functions if, given any nonempty open set $\omega \subset \T$,   the only solution $u\in X$ of \eqref{A1} fulfilling
\[
u(x,t)=0 \qquad \text{ for }(x,t)\in \omega\times (0,T),
\]
is the trivial one $u\equiv 0$. Such a property is very important in Control Theory, as it is equivalent to the approximate controllability for linear PDE, and it is involved in
the classical uniqueness/compactness approach in the proof of the stability for a PDE with a localized damping.
The UCP is usually proved with the aid of some Carleman estimate (see e.g. \cite{SS}).
The UCP for KdV was  established in \cite{Zhang92} by the inverse scattering approach, in \cite{EKPV,RZ2006,SS} by means of Carleman estimates, and in \cite{Bourgain}
by a perturbative approach and Fourier analysis. For BBM, the study of the UCP is only at its early age. The main reason is that both $x=const$ and $t=const$ are characteristic lines for \eqref{A1}.
Thus, the Cauchy problem in the UCP (assuming e.g. that $u=0$ for $x\le 0$, and solving BBM for $x\ge 0$) is characteristic, which prevents from applying Holmgren's theorem, even for
the linearized equation. The Carleman approach for the UCP of BBM was developed in \cite{DPM} and in \cite{Yamamoto}. Unfortunately, Theorems
3.1-3.4 in \cite{DPM} are not correct without further assumptions, as noticed in \cite{ZZ}. On the other hand, the UCP in \cite{Yamamoto} for the BBM-like equation
\[
u_x-u_{txx}=p(x,t)u_x+q(x,t)u,\qquad x\in (0,1),\ t\in (0,T),
\]
where $p\in L^\infty (0,T;L^\infty (0,1))$ and $q\in L^\infty (0,T;L^2(0,1))$, requires $u(1,t)=u_x(1,t)=0$ for $t\in (0,T)$ and
\be
\label{CI}
u(x,0)=0\qquad \text{ for } x\in (0,1).
\ee
(Note, however, that nothing is required for $u(0,t)$.)
Because of \eqref{CI}, such a UCP cannot be used for the stabilization problem. More can be said for a linearized BBM equation with potential functions depending only on $x$.
It was proved in \cite{Micu} that the only solution $u\in C([0,T],H^1(0,1))$ of the linearized BBM equation
\ba
&& u_t-u_{txx}+u_x=0,\qquad x\in (0,1),\  t\in (0,T), \label{micu1}\\
&& u(0,t)=u(1,t)=0, \qquad t\in (0,T)\label{micu2}
\ea
fulfilling
$u_x(1,t)=0$ for all  $t\in (0,T)$
is the trivial one $u\equiv 0$. It is worth noticing that the proof of that result strongly used  the fact that the solutions of \eqref{micu1}-\eqref{micu2} are {\em analytic in time}.
On the other hand, several difficult UCP results based on spectral analysis are given in \cite{ZZ} for the system
\ba
&&u_t-u_{txx}=[\alpha (x)u]_x+\beta (x)u,\qquad x\in (0,1),\ t\in (0,T), \label{ZZ1}\\
&&u(0,t)=u(1,t)=0,\qquad t\in (0,T).\label{ZZ2}
\ea
As noticed in \cite{ZZ}, the UCP fails for \eqref{ZZ1}-\eqref{ZZ2} whenever both $\alpha$ and $\beta$ vanish on some open set $\omega \subset \T$, so that
the UCP  depends not only on the regularity of the functions $\alpha$ and $\beta$, but also on their zero sets.
Bourgain's approach \cite{Bourgain}  for the UCP of KdV (or NLS) is based on the fact that the Fourier transform of a compactly supported function extends to an entire function of  exponential type.
%the following idea: a solution $u=u(x,t)$ of KdV on $\R$ which vanishes for say $|x|>L$ has a Fourier transform
%in $x$ that can be extended as an entire function $\hat u(\xi ,t)$ of exponential type at most $L$, by virtue of Paley-Wiener theorem. For the linearized equation, it is given explicitly as
%\[
%\hat u(\xi, t) =e^{i(\xi ^3-\xi)t}\hat(\xi ,0), \qquad \xi\in\R, \ t\in\R.
%\]
%The above relation is still true if $\xi$ is replaced by $\xi + i\sigma$, by analytic continuation. This
%Thus, if $u(x,t)=0$ for $|x|>L$ and $t\in (0,T), we
The proof of the UCP in \cite{Bourgain}  rests on estimates at high frequencies using the intuitive property  that the nonlinear term in Duhamel formula is perturbative.
As noticed in \cite{Mammeri}, that argument does
not seem to be applicable to BBM. Actually, if we follow Bourgain's idea for the linearized BBM equation
\be
\label{A4}
u_t - u_{txx}  + u_x=0
\ee
on $\R$, and assume that some solution $u$ vanishes for $|x|>L$ and $t\in (0,T)$, then its Fourier transform in $x$, denoted by $\hat u(\xi,t)$, is readily found to be
\[
\hat u(\xi, t) =\exp (\frac{-it\xi}{\xi ^2+1})\hat u(\xi, 0),\qquad \xi\in\R,\ t\in (0,T).
\]
The consideration of high frequencies is useless here.
By analytic continuation, the above equation still holds for all $\xi=\xi _1+i\xi _2\in \C \setminus \{ \pm i\} $. 
Picking any $t>0$, $\xi_1=0$ and letting $\xi_2\to 1^-$, we readily infer that $\partial _\xi ^n \hat u(i,0)=0$ for
 all $n\ge 0$, so that $\hat u(.,0)\equiv 0$ and hence $u\equiv 0$. Note that
 \be
 \partial _\xi ^n \hat u(i,t)=\int_{-\infty}^\infty u(x,t)(-ix)^n e^xdx,
 \ee
 and that it can be shown by induction on $n$ that all the moments $M_n(t)=\int_{-\infty}^\infty u(x,t)x^ne^xdx$ vanish on $(0,T)$, so that $u\equiv 0$.  Unfortunately, we cannot  modify
 the above argument to deal with the UCP for the full BBM equation, as the nonlinear term has no reason to be perturbative at the ``small''  frequencies $\xi =\pm i$.
 We point out that  a moment approach, inspired by \cite{Constantin}, was nevertheless applied in \cite{Mammeri} to prove the UCP for some KP-BBM-II equation.

In this paper, we shall apply the moment approach to prove the UCP for a generalized BBM equation
\[
u_t - u_{txx} +[f(u)]_x=0,
\]
where $f:\R\to\R $ is smooth and {\em nonnegative}. The choice $f(u)=u^2/2$ gives the so-called
Morrison-Meiss-Carey (MMC) equation (also called {\em equal width wave equation}, see \cite{HESG,MMC}).
Incorporating a localized damping in the above equation, we obtain the equation
\[
u_t-u_{txx}+[f(u)]_x+a(x)u=0, \qquad x\in\T,
\]
whose solutions are proved to tend weakly to $0$ in $H^1(\T)$ as $t\to\infty$. Note that similar results were proved in \cite{LV}
with a boundary dissipation.

Bourgain's approach, in its complex analytic original form, can be used to derive the UCP for the following BBM-like equation
\[
u_t-u_{txx}+u_x+(u*u)_x=0
\]
in which the (nonlocal) term $(u*u)_x$ is substituted to the classical nonlinear term $uu_x$ in BBM.

For the original BBM equation   \eqref{A1}, we shall derive a UCP for solutions issuing from initial data that are small enough in $H^1(\T)$ and with nonnegative
mean values. The proof, which is very reminiscent of La Salle invariance principle,  will combine the analyticity in time of solutions of BBM, the existence of three invariants of motion, and the use of
some appropriate Lyapunov function.

The second part of this work is concerned with the control of the BBM equation.
Consider first the linearized BBM equation with a control force
\be
\label{C30}
u_t -u_{txx} +u_x = a(x)h(x,t),
\ee
where $a$ is supported in some subset of $\T$ and $h$ stands for the control input.
It was proved in \cite{Micu,ZZ} that \eqref{C30}
is  {\em approximatively controllable} in $H^1(\T)$. It turns out that \eqref{C30}  is {\em  not exactly controllable} in $H^1(\T )$ \cite{Micu}.
This is in sharp contrast with the good control properties of other dispersive equations (on periodic domains, see e.g. \cite{LRZ,RussellZhang96} for KdV,
\cite{DGL,Laurent1,Laurent2,RZ2007b,RZ2009}  for the nonlinear Schr\"odinger equation, \cite{LO,LR} for the Benjamin-Ono equation,
\cite{MORZ} for Boussinesq system, and \cite{Glass} for Camassa-Holm equation).
The bad control properties of \eqref{C30} come from the existence of a limit point in the spectrum. Such a phenomenon was noticed
in \cite{Russell} for the beam equation with internal damping, in \cite{LS} for the plate equation with internal
damping, in \cite{Micu} for the linearized BBM equation, and more recently in \cite{RR} for the wave equation with structural damping.

It is by now classical that an ``intermediate'' equation between \eqref{A1} and \eqref{A2} can be derived from \eqref{A1} by working in a moving frame
$x=-ct$ with $c\in \R\setminus \{ 0\} $. Indeed, letting
\be
\label{A5}
v(x,t)=u(x-ct,t)
\ee
we readily see that \eqref{A1} is transformed into the following KdV-BBM equation
\be
\label{A6}
v_t+(c+1)v_x-cv_{xxx}-v_{txx}+vv_x=0.
\ee
%These KdV-BBM-type equations prove to be useful in numerical studies because of the stability in the numerical schemes provided by the
%BBM term $-y_{txx}$ (the KdV term $y_{xxx}$ causes instabilities propagating to the left).
It is then reasonable to expect the control properties of \eqref{A6} to be better than those of \eqref{A1}, thanks to the KdV term $-cv_{xxx}$ in \eqref{A6}.
We shall prove that the equation \eqref{A6} with a forcing term $a(x)k(x,t)$ supported in (any given) subdomain is locally exactly controllable in $H^1(\T)$
in time $T>(2\pi)/|c|$. Going back to the original variables, it means that the equation
\be
\label{A7}
u_t + u_x -u_{txx} + uu_x = a(x+ct) h(x,t)
\ee
with a moving distributed control is exactly controllable in $H^1(\T )$ in (sufficiently) large time. Actually, the control time is chosen in such a way that
the support of the control, which is moving at the constant velocity $c$,  can visit all the domain $\T$. Using the same idea, it has been proved recently in \cite{MRR}
that the wave equation with structural damping is null controllable in large time when controlled with a moving distributed control.

The concept of  moving point control was introduced by J. L. Lions in \cite{Lions} for the wave equation. One important motivation for this kind of control is that the exact controllability
of the wave equation with a pointwise control and Dirichlet boundary conditions fails if the point is a zero of some eigenfunction of the Dirichlet Laplacian, while it holds
 when the point is moving under some conditions easy to check
(see e.g. \cite{Castro}).
The controllability of the wave equation (resp. of the heat equation) with a moving point control was investigated in \cite{Castro,Khapalov1,Lions} (resp. in \cite{CZ1,Khapalov2}).

Thus, the appearance of the KdV term $-cv_{xxx}$ in \eqref{A6} results in much better control properties.
We shall see that
\begin{enumerate}
\item[(i)] there is no limit point in the spectrum of the linearized KdV-BBM equation, which is of ``hyperbolic'' type;
\item[(ii)]  a UCP for the full KdV-BBM equation can be derived from Carleman estimates for a system of coupled elliptic-hyperbolic equations.
\end{enumerate}
It follows that  one can expect a semiglobal exponential stability when applying a localized damping with a moving support.
We will see that this is indeed the case. Combining the local exact controllability  to the semiglobal exponential stability result, we obtain the following theorem which is the main
result of the paper.
\begin{theorem}
\label{globalcontrollability} Assume  given $a\in C^\infty(\T )$ with
$a\ne 0$ and  $c\in \R \setminus \{ 0 \} $.  Let $s\geq 1$ and  $R >0$ be given.  Then there exists a 
time  $T=T(s,R) > 2\pi / |c|$ such that for any
$u_0,u_T\in H^s(\T)$ with 
\be  
||u_0||_{H^s}\le R, \qquad ||u_T||_{H^s} \le R, 
\ee 
there exists a control $h\in L^2(0,T;H^{s-2}(\T ))$ 
such that the solution $u\in C([0,T];H^s(\T))$ of
\begin{eqnarray*}
&&u_t-u_{txx}+ u_x+uu_x=a(x+ct)h(x,t),\quad x\in \T,\  t\in (0,T)\\
&&u(x,0)=u_0(x),\quad x\in\T
\end{eqnarray*}
satisfies
\[
u(x,T)=u_T(x),\quad x\in \T.
\]
\end{theorem}

The paper is scheduled as follows. In Section \ref{sec:GWP} we recall some useful facts (global well-posedness, invariants of motion, time analyticity) about BBM. In Section \ref{sec:UCP} we establish
the UCP for BBM. In Section \ref{sec:UCP2}  we prove the UCP for other BBM-like equations, including the MMC equation
and the BBM equation with a nonlocal term. Section \ref{sec:UCPKdVBBM} is concerned with the UCP
for the KdV-BBM equation. The KdV-BBM equation is first split into a coupled system of an elliptic equation and a transport equation. Next, we prove some Carleman estimates with the same singular weights
for both the elliptic and the hyperbolic equations, and we derive the UCP for KdV-BBM by combining these Carleman estimates with a regularization process. Those results are used in Section
\ref{sec:control} to prove the exact controllability of KdV-BBM and the semiglobal exponential stability of the same equation with a localized damping term.

\section{Wellposedness, analyticity in time and invariants of motion}
\label{sec:GWP}
Throughout  the paper, for any $s\ge 0$, $H^s(\T)$ denotes the Sobolev space
\[
H^s(\T )=\{ u:\T \to \R; \ ||u||_{H^s}:=||(1-\partial _x ^2)^{\frac{s}{2}} u||_{L^2(\T )}  <\infty \}. 
\]
Its dual is denoted $H^{-s}(\T)$.
 
Let us consider the initial value problem (IVP)
\ba
&&u_t-u_{txx}+ u_x+uu_x=0,\quad x\in \T,\  t\in \R \label{B1}\\
&&u(x,0)=u_0(x). \label{B2}
\ea
Let $A = -(1-\partial _x^2 )^{-1}\partial _x \in {\mathcal L} (H^s(\T ), H^{s+1}(\T ))$ (for any $s\in \R$) and $W(t)=e^{tA}$ for $t\in \R$.
We put \eqref{B1}-\eqref{B2} in its integral form
\be
\label{B3}
u(t) = W(t) u_0  + \int_0^t W(t-s) A (u^2/2)(s)ds.
\ee
For $s\ge 0$ and $T>0$, let
\[
X_T^s=C([-T,T];H^s(\T )).
\]
Note that for $u\in X_T^s$, $u$ solves \eqref{B1} in ${\mathcal D}'(-T,T;H^{s-2}(\T))$ and \eqref{B2} if, and only if, it fulfills \eqref{B3} for all  $t\in [-T,T]$.
The following result will be used thereafter.
\begin{theorem} (\cite{BT,Roumegoux})
\label{thmA}
Let $s\ge 0$, $u_0\in H^s(\T )$ and $T>0$. Then there exists a unique solution $u\in X_T^s$ of \eqref{B1}-\eqref{B2}
(or, alternatively, \eqref{B3}).  Furthermore, for any $R>0$, the map $u_0\mapsto u$  is real analytic from $B_R (H^s(\T ))$ into $X_T^s$.
\end{theorem}
Some additional properties are collected in the following
\begin{proposition}
\label{analytic}
For $u_0\in H^1(\T )$, the solution $u(t)$ of the IVP \eqref{B1}-\eqref{B2} satisfies $u\in C^\omega (\R ;H^1(\T ))$. Moreover the three integral terms
$\int_{\T}u\, dx$, $\int_{\T} (u^2+u_x^2)dx$ and $\int_{T} (u^3+3u^2)dx$ are invariants of motion (i.e., they remain constant over time).
\end{proposition}
\begin{proof}
Let us begin with the invariants of motion.
For $u_0\in H^1(\T )$, $u\in X_T^1$ for all $T>0$, hence
\[
u_t = -(1-\partial _x ^2) ^{-1} \partial _x (u+\frac{u^2}{2}) \in X_T^2.
\]
Therefore, all the terms in \eqref{B1} belong to $X_T^0$. Scaling in \eqref{B1} by $1$ (resp. by $u$) yields after some integrations by parts
\[
\frac{d}{dt} \int_{\T } u\, dx =0 \qquad (\text{resp. } \ \frac{d}{dt} \int_{\T} (u^2+u_x^2)dx =0.)
\]
For the last invariant of motion, we notice (following \cite{Olver}) that
\[
(\frac{1}{3}(u+1)^3)_t - (u_t^2-u_{xt}^2+(u+1)^2u_{xt} -\frac{1}{4} (u+1)^4)_x=0.
\]
Integrating on $\T$ yields $(d/dt)\int_T(u+1)^3dx=0$. Since $(d/dt)\int_{\T} (3u+1)dx =0$, we infer that
\[
\frac{d}{dt} \int_{\T} (u^3+3u^2)dx=0.
\]
Let us now prove that $u\in C^\omega (\R ; H^1(\T ))$. Since $u\in C^1(\R ;  H^1 (\T ) )$, it is sufficient to check that for any $u_0\in H^1(\T ) $ there are some numbers
$b>0$, $M>0$, and some sequence $(u_n)_{n\ge 1}$ in $H^1(\T )$ with
\be
\label{P1}
||u_n||_{H^1} \le \frac{M}{b^n},\qquad n\ge 0,
\ee
such that
\be
\label{P2}
u(t) =\sum_{n\ge 0} t^nu_n,\qquad t\in (-b,b).
\ee
Note that the convergence of the series in \eqref{P2} holds in $H^1(\T )$ uniformly on $[-rb, rb]$ for each $r<1$.
Actually, we prove that $u$ can be extended as an analytic function from $D_b:=\{ z\in \C;\  |z| < b \}$ into the space
$H^1_\C (\T ):=H^1(\T ; \C)$,  endowed with the Euclidean norm
\[
|| \sum_{k\in \Z}  {\hat u}_k e^{ikx} ||_{H^1} =  ( \sum_{k\in \Z} (1+|k|^2) |{\hat u}_k |^2 )^{\frac{1}{2}} .
\]
We adapt the classical proof of the analyticity of the flow for an ODE with an analytic
vector field (see e.g. \cite{Hochschild}) to our infinite dimensional framework.
For  $u\in H^1_\C (\T )$, let $Au=-(1-\partial _x ^2)^{-1} \partial _x u$ and $f(u)=A(u+u^2)$. Since
$|k|\le (k^2 + 1)/2$ for all $k\in \Z$, $||A||_{ {\mathcal L} (H^1_\C (\T ))} \le 1/2$. Pick a positive constant  $C_1$ such that
\[
||u^2||_{H^1} \le C_1 ||u||^2_{H^1}\qquad \text{ for all } u\in H^1_\C (\T ).
\]
We define by induction on $q$ a sequence $(u^q)$ of analytic functions from $\C$ to $H^1_\C (\T)$ which will converge
uniformly on $D_T$, for $T>0$ small enough, to a solution of the integral equation
\[
u(z)= u_0 + \int_{[0,z]} f(u(\zeta ))d\zeta = u_0 + \int_0^1 f(u(sz))zds.
\]
Let
\begin{eqnarray*}
u^0(z)        &=& u_0, \qquad \text{ for } z\in \C \\
u^{q+1}(z) &=& u_0 + \int_{[0,z]} f(u^q(\zeta ))d\zeta ,\qquad \text{ for } q\ge 0,\ z\in \C .
\end{eqnarray*}
{\sc Claim 1.}  $u^q(z)=\sum_{n\ge 0} z^n v_n^q$ for all $z\in\C$ and some sequence $(v_n^q)$ in $H^1_\C (\T )$ with
\[
||v_n^q||_{H^1} \le \frac{M(q,b)}{b^n}\qquad \text{ for all } q,n\in \N, \ b>0.
\]
The proof of Claim 1 is done by induction on $q\ge 0$. The result is clear for $q=0$ with $M(0,b)=||u_0||_{H^1}$, since
$v_0^0=u_0$ and $v_n^0=0$ for $n\ge 1$.
Assume Claim 1 proved for some $q\ge 0$. Then, for any $r\in (0,1)$ and any $b>0$
\[
||z^n v_n^q||_{H^1} \le M(q,b)  r^n\qquad \text{ for } |z|\le rb,
\]
so that the series $\sum_{n\ge 0} z^n v_n^q$ converges absolutely in $H^1_\C (\T )$ uniformly for $z\in \overline{D_{rb}}$. The same holds true for the series
$\sum_{n\ge 0} z^n (\sum_{0\le l\le n}v_l^qv_{n-l}^q )$. It follows that
\[
f(u^q(\zeta ))=A\left( \sum_{n\ge 0}\zeta ^n v_n^q + \sum _{n\ge 0} \zeta ^n (\sum_{0\le l\le n} v_l^q v_{n-l}^q ) \right)
\]
converges uniformly for $\zeta\in \overline{D_{rb}}$. Thus
\begin{eqnarray*}
u^{q+1}(z)  &=& u_0 + \int_{[0,z]} \sum _{n\ge 0} \zeta ^n A (v_n^q + \sum_{0\le l\le n} v_l^q v_{n-l}^q)d\zeta \\
%&=& u_0 + \sum_{n\ge 0} \frac{z^{n+1}}{n+1} A(v_n^q +  \sum_{0\le l \le n } v_l^qv_{n-l}^q )\\
&=& \sum_{n\ge 0} z^n v_n^{q+1}
\end{eqnarray*}
where
\begin{eqnarray*}
v_0^{q+1}&=&u_0,\\
v_n^{q+1}&=&\frac{1}{n}A(v_{n-1}^q + \sum_{0\le l \le n-1}v_l^qv_{n-1-l}^q) \qquad \text{ for }\  n\ge 1.
\end{eqnarray*}
It follows that for $n\ge 1$
\[
||v_n^{q+1} ||_{H^1} \le \frac{||A||}{n} (\frac{M(q,b)}{b^{n-1}} + n C_1 \frac{M^2(q,b)}{b^{n-1}} ) \le \frac{M(q+1,b)}{b^n}
\]
with
\[
M(q+1,b):= \sup\{||u_0||_{H^1}, b||A||(M(q,b)+ C_1M^2(q,b)\} .
\]
Claim 1 is proved.\\
{\sc Claim 2.} Let $T:= (2||A||(1+4C_1||u_0||_{H^1}))^{-1}$. Then $||u^q-u||_{L^\infty ( \overline{D_T}; H^1_\C (\T ))} \to 0 $ as $q\to \infty$ for some
$u\in C(\overline{D_T};H^1_\C (\T ))$. \\
Let $Z_T=C(\overline{D_T}; H^1_\C (\T ))$ be endowed with the norm $|||v|||=\sup_{|z|\le T} ||v(z)||_{H^1}$. Let $R>0$, and for $v\in B_R:=\{ v\in Z_T; \ ||| v ||| \le R\} $, let
\[
(\Gamma v) (z) = u_0 + \int_{ [0,z] } f(v (\zeta ))\, d\zeta .
\]
Then
\begin{eqnarray*}
||| \Gamma v ||| &\le& ||u_0||_{H^1} + T ||A|| ( |||v||| + C_1 ||| v |||^2) \le ||u_0||_{H^1} + T ||A|| (R+C_1 R^2),\\
||\Gamma v_1-\Gamma v_2 ||| &\le & T ||A|| (|||v_1-v_2||| + |||v_1^2-v_2^2|||) \le T ||A|| (1+2C_1R) |||v_1-v_2|||.
\end{eqnarray*}
Pick $R=2||u_0||_{H^1}$ and $T=(2||A|| (1+2C_1R))^{-1}$. Then $\Gamma$ contracts in $B_R$. The sequence
$(u^q)$, which is given by Picard iteration scheme, has a limit $u$ in $Z_T$ which fulfills
\[
u(z)= u_0 + \int_{[0,z]} f(u(\zeta ))d\zeta, \qquad |z|\le T.
\]
In particular, $u\in C^1 ([-T,T];H^1(\T))$ (the $u^q(z)$ being real-valued for $z\in \R$)
and it satisfies $u_t=f(u)$ on $[-T,T]$ together with $u(0)=u_0$; that is, $u$ solves \eqref{B1}-\eqref{B2}
in the class $C^1([-T,T];H^1(\T))\subset X_T^1$.\\
{\sc Claim 3.} $u(z)= \sum_{n\ge 0} z^n v_n$ for $|z|<T$, where $v_n=\lim_{q\to \infty} v_n^q$ for each $n\ge 0$. \\
From Claim 1, we infer that for all $n\ge 1$
\[
v_n^q = \frac{1}{2\pi i} \int_{|z|= T} z^{-n-1} u^q(z)\, dz,
\]
hence
\[
||v_n^p - v_n^q||_{H^1} \le T^{-n} |||u^p-u^q|||.
\]
From Claim 2, we infer that $(v_n^q)$ is a Cauchy sequence in $H^1_\C (\T )$. Let $v_n$ denote its limit in $H^1_\C ( \T)$. Note that
\[
||v_n -v_n^q||_{H^1} \le  T^{-n} ||| u -u^q |||,
\]
and hence the series $\sum_{n\ge 0 } z^n v_n$ is convergent for $|z|<T$. Therefore, for $|z|\le rT$ with $r<1$,
\[
||\sum_{n\ge 0} z^n (v_n - v_n^q) ||_{H^1} \le (1-r)^{-1}  ||| u - u^q|||,
\]
and hence
$u^q(z)=\sum_{n\ge 0} z^nv_n^q\to \sum_{n\ge 0} z^nv_n$ in $Z_{rT}$ as $q\to \infty$. It follows that
\[
u(z)=\sum_{n\ge 0} z^n v_n\qquad \text{ for } |z| <T.
\]
The proof of Proposition \ref{analytic} is complete.
\end{proof}
\section{Unique Continuation Property for BBM}
\label{sec:UCP}
In this section we prove a UCP for the BBM equation for small solutions with nonnegative mean values.
\begin{theorem}
\label{UCP1}
Let $u_0\in H^1(\T )$ be such that
\ba
\int_{\T } u_0(x) dx &\ge&  0, \label{U1}\\
\text{and} \qquad ||u _0||_{L^\infty (\T )} &<& 3. \label{U2}
\ea
Assume that the solution $u$ of \eqref{B1}-\eqref{B2} satisfies
\begin{equation}
u(x,t) =0 \qquad \text{ for all } \ (x,t)\in \omega \times (0,T), \label{U3}
\end{equation}
where $\omega\subset \T$ is a nonempty open set and $T>0$. Then $u_0=0$, and hence $u\equiv 0$.
\end{theorem}
\begin{proof} We identify $\T$ to $(0,2\pi )$ in such a way that  $\omega
\supset (0,\varepsilon ) \cup (2\pi -\varepsilon , 2\pi )$ for some $\varepsilon >0$. Since
$u \in C^\omega (\R ;H^1(\T ))$ by Proposition \ref{analytic}, we have that $u(x,.)\in C^\omega (\R )$ for all $x\in \T$.
\eqref{U3} gives then that
\begin{equation}
u(x,t) = 0 \quad \text{ for } (x,t)\in \omega \times \R . \label{U4}
\end{equation}
Introduce the function
\[
v(x,t) = \int_0^x u(y,t) dy.
\]
Then $v\in C^\omega (\R ;H^2(0, 2\pi  )) $ and $v$ satisfies
\begin{equation}
v_t -v_{txx} +v_x + \frac{u^2}{2} =0, \qquad x\in (0,2\pi ),  \label{U5}
\end{equation}
as it may be seen by integrating \eqref{B1} on $(0,x)$. Let
\[
I(t) = \int_0^{2\pi}  v(x,t) dx.
\]
Note that $I\in C^\omega (\R )$. Integrating  \eqref{U5} on $(0,2\pi ) $ gives with \eqref{U1}
\[
I_t =  -\int_0^{2\pi} u_0(x) dx -\frac{1}{2}\int_0^{2\pi} |u(x,t)|^2 dx \le 0.
\]
Since $||u(t)||_{H^1}=||u_0||_{H^1}$ for all $t\in\R$, $v\in L^\infty (\R , H^2( 0, 2\pi  ))$ and $I\in L^\infty (\R )$. It follows that
the function $I$ has a finite limit as $t\to \infty$, that we denote by $l$. From the boundedness of $||u(t)||_{H^1(\T )}$   for $t\in\R$,
we infer the existence of a sequence $t_n\nearrow +\infty$ such that
\begin{equation}
u(t_n)\cf {\tilde u}_0 \qquad \text{ in } H^1(\T )\label{U8}
\end{equation}
for some $\tilde u_0\in H^1(\T )$. Let $\tilde u$ denote the solution of the IVP for BBM corresponding to the initial
data ${\tilde u}_0$; that is,  $\tilde u$ solves
\begin{eqnarray*}
&&\tilde u_t -\tilde u_{txx} +\tilde u_x + \tilde u\tilde u_x =0, \qquad x\in\T, \ t\in \R , \\
&&\tilde u(x,0) =\tilde u_0(x).
\end{eqnarray*}
Pick any $s\in (1/2,1)$. As $u(t_n)\to \tilde u_0$ strongly in $ H^s (\T )$, we infer from Theorem \ref{thmA} that
\be
u(t_n+\cdot ) \to \tilde u \quad \text{ in } C([0,1];H^s(\T )). \label{U11}
\ee
It follows from \eqref{U4}, \eqref{U11} and the fact that $\tilde u\in C^\omega (\R , H^1(\T ))$ that
\begin{equation*}
\tilde u(x,t) = 0 \qquad \text{ for } (x,t)\in \omega \times \R.
\end{equation*}
On the other hand, $\int_0^{2\pi} \tilde u_0(x)dx=\int _0^{2\pi} u_0(x)dx$ from \eqref{U8} and the invariance of $\int_0^{2\pi} u(x,t)dx$.
Let $\tilde v(x,t) =\int_0^x \tilde u (y,t) dy$ and $\tilde I (t) = \int_0^{2\pi}  \tilde v (x,t) dx $. Then we still have that
\be
\tilde I _t = -\int_0^{2\pi} u_0(x) dx -\frac{1}{2}\int_0^{2\pi} |\tilde u (x,t)|^2 dx \le 0. \label{U12}
\ee
But we infer from \eqref{U11} that
\begin{eqnarray*}
I(t_n)\to \tilde I(0),\qquad I(t_n +1)\to \tilde I(1).
\end{eqnarray*}
Since
\[
\lim_{n\to\infty} I(t_n) = \lim_{n\to \infty} I(t_n +1)=l,
\]
we have that $\tilde I(0)=\tilde I(1)$. Combined to \eqref{U12}, this yields
\[
\tilde u(x,t) =0 \quad  (x,t)\in \T \times [0,1].
\]
In particular, $\tilde u_0=0$. From \eqref{U8}, we infer that
\[
\int_0^{2\pi}  ( u^3(x,t_n) + 3 u^2(x,t_n) ) dx \to 0\quad  \text{ as } n\to \infty .
\]
As $\int_0^{2\pi}  (u^3+3u^2)dx$ is a conserved quantity, we infer  that
\[
\int_0^{2\pi}   (3+u_0(x))\, |u_0(x)|^2dx=0,
\]
which, combined to \eqref{U2}, yields $u_0=0$.
\end{proof}
\begin{remark}
Note that Theorem \ref{UCP1} is false if the assumptions $u_0\in H^1(\T)$ and \eqref{U1} are removed. Indeed, if $u\in C(\R; L^2(\T))$ is defined for
$x\in \T\sim (0,2\pi)$ and $t\in \R$ by
\[
u(x,t)=u_0(x)=
\left\{
\begin{array}{ll}
-2 \quad &\text{if }\  |x-\pi |\le  \frac{\pi}{2}, \\
0              &\text{if }\  \frac{\pi}{2} < |x-\pi | < \pi ,
\end{array}
\right.
\]
then \eqref{B1} and \eqref{B2} are satisfied, although $u\not\equiv 0$.
\end{remark}
%\subsection{Stabilization with a localized (fixed) damping}
\section{Unique Continuation Property for BBM-like equations}
\label{sec:UCP2}
We shall consider BBM-like equations with different nonlinear terms. We first consider a generalized BBM equation without drift term, and next
a BBM-like  equation with  a nonlocal bilinear term.
%%%%%%%%%%%%%%%%%%%%%%%%%%%%%%%%%%%%%%%%%%%%%%%%%%%%%%%%%%%%%%%%%%%%%%%
\subsection{Generalized BBM equation without drift term}
We consider the following generalized BBM equation
\ba
&&u_t-u_{txx} +[f(u)]_x =0,\qquad x\in \T , \ t\in \R  \label{V1}\\
&&u(x,0)= u_0(x), \label{V2}
\ea
where $f\in C^1(\R )$, $f(u)\ge 0$ for all $u\in \R$, and the only solution $u\in (-\delta ,\delta )$ of $f(u)=0$ is $u=0$, for some number $\delta >0$.
That class of BBM-like equations includes the Morrison-Meiss-Carey equation %(also called {\em equal width wave equation}, see \cite{HESG})
\[
u_t - u_{txx} +uu_x=0
\]
for $f(u)=u^2/2$. Note that the global wellposedness of \eqref{V1}-\eqref{V2} in $H^1(\T )$  can easily be derived from the contraction mapping theorem
and the conservation of the  $H^1$-norm.   It turns out that the UCP can be derived in a straight way and without any additional assumption on the initial data.
\begin{theorem}
\label{thm4}
Let $f$ be as above, and let $\omega$ be a nonempty open set in $\T$. Let $u_0\in H^1(\T )$ be such that the solution $u$ of \eqref{V1}-\eqref{V2} satisfies
$u(x,t)=0$ for $(x,t)\in \omega \times (0,T)$ for some $T>0$. Then $u_0=0$.
\end{theorem}
\begin{proof}
Once again, we can assume without loss of generality that $\omega =(0,\varepsilon ) \cup (2\pi -\varepsilon , 2\pi)$. The prolongation of $u$ by 0 on
$(\R \setminus (0,2\pi ))\times (0,T)$, still denoted by $u$, satisfies
\ba
u_t-u_{txx} +[f(u)]_x =0,&& x\in \R,\ t\in (0,T)  \label{V5} \\
u(x,t) =0,&&  x\not\in (\varepsilon, 2\pi -\varepsilon ) , \ t\in (0,T) \label{V6}\\
u\in C([0,T];H^1(\R )),&& u_t\in C([0,T];H^2(\R )). \label{V7}
\ea
Scaling in \eqref{V5} by $e^x$ yields for $t\in (0,T)$
\[
\int_{-\infty}^\infty f(u(x,t))e^x dx=0,
\]
for $\int_{-\infty }^\infty u_{txx}e^x dx = \int_{-\infty}^\infty u_te^xdx$ by two integrations by parts. Since $f$ is nonnegative, this yields
\[
f(u(x,t))=0\qquad \text{ for } (x,t)\in \R \times (0,T).
\]
Since $u$ is continuous and it vanishes for $x\not\in (\varepsilon , 2\pi -\varepsilon)$, we infer from the assumptions about $f$ that $u\equiv 0$.
\end{proof}
Pick any nonnegative function $a\in C^\infty ( \T )$ with $\omega := \{ x\in\T;\ a(x) > 0\}$ nonempty.  We are interested in the stability properties of the system
\ba
&&u_t-u_{txx} +[f(u)]_x +a(x)u=0,\qquad x\in \T ,\  t\ge 0    \label{K11}\\
&&u(x,0)= u_0(x), \label{K12}
\ea
where $f$ is as above. The following weak stability result holds.
\begin{corollary}
\label{cor2}
Let $u_0\in H^1(\T)$. Then \eqref{K11}-\eqref{K12} admits a unique solution $u\in C([0,T];H^1(\T))$ for all $T>0$. Furthermore, $u(t)\to 0$ weakly in $H^1(\T)$, hence strongly in
$H^s(\T )$ for $s<1$, as $t\to +\infty$.
\end{corollary}
\begin{proof}
The local wellposedness in $H^s(\T )$ for any $s> 1/2$ is derived from the contraction mapping theorem in much the same way as for Theorem \ref{thmA}. The global wellposedness
in $H^1(\T)$ follows at once from the energy identity
\be
\label{K13}
||u(T)||^2_{H^1 } - ||u_0||^2_{H^1 } + 2\int_0^T\!\!\!\int_{\T} a(x)|u(x,t)| ^2dxdt =0.
\ee
obtained by scaling each term in \eqref{K11} by $u$. On the other hand, still from the application of the contraction mapping theorem, given any $s >1/2$, any $\rho >0$  and any $u_0,v_0\in H^s(\T )$ with
$||u_0||_{H^s(\T )}\le \rho$, $||v_0||_{H^s(\T)}\le \rho$, there is some time $T=T(s,\rho )>0$ such that
the solutions $u$ and $v$ of \eqref{K11}-\eqref{K12} corresponding to the initial data $u_0$ and $v_0$, respectively,  fulfill
\be
\label{K14}
||u - v||_{C([0,T]; H^s(\T ))} \le 2 ||u_0-v_0||_{H^s(\T )}.
\ee
Pick any initial data $u_0\in H^1(\T )$, any $s\in (1/2,1)$, and let $\rho = ||u_0||_{H^1(\T)}$ and $T=T(s,\rho)$. Note that $||u(t)||_{H^1}$ is nonincreasing by \eqref{K13}, hence it has
a nonnegative limit $l$ as $t\to \infty$. Let
$v_0$ be in the $\omega-$limit set of $(u(t))_{t\ge 0}$ in $H^1(\T)$ for the weak topology; that is, for some sequence $t_n\to \infty$ we have $u(t_n)\to v_0$ weakly in $H^1(\T )$. Extracting a subsequence if needed,
we may assume that  $t_{n+1}-t_n\ge T$ for all $n$. From \eqref{K13} we infer that
\be
\label{K15}
\lim_{n\to \infty} \int_{t_n}^{t_{n+1}}\!\!\!\int_{\T} a(x) |u(x,t)|^2 dxdt=0.
\ee
Since $u(t_n)\to v_0$ (strongly) in $H^s(\T )$, and $||u(t_n)||_{H^s(\T )} \le ||u(t_n)||_{H^1(\T )} \le \rho$, we have from \eqref{K14} that
\be
\label{K16}
u(t_n+\cdot ) \to v\qquad \text{ in } C([0,T];H^s(\T ))\quad \text{ as } n\to \infty ,
\ee
where $v=v(x,t)$ denotes the solution of
\begin{eqnarray*}
&&v_t-v_{txx} +[f(v)]_x +a(x)v=0,\qquad x\in \T ,\ t\ge 0, \\
&&v(x,0)= v_0(x).
\end{eqnarray*}
Note that $v\in C([0,T];H^1(\T ))$ for $v_0\in H^1(\T ))$. \eqref{K15} combined to \eqref{K16} yields
\[
\int_{0}^{T}\!\!\!\int_{\T} a(x) |v(x,t)|^2 dxdt=0,
\]
so that $av\equiv 0$. By Theorem \ref{thm4}, %$v\equiv 0$ on $\T \times (0,T)$. It follows that
$v_0=0$ and hence, as $t\to \infty$,
\begin{eqnarray*}
u(t)\to 0  &&\text{weakly in } H^1(\T),\\
u(t)\to  0 &&\text{strongly in } H^s(\T ) \text{ for } s<1.
\end{eqnarray*}
\end{proof}
%%%%%%%%%%%%%%%%%%%%%%%%%%%%%%%%%%%%%%%%%%%%%%%%%%%%%%%%%%%%%%%%%%%%%%
\subsection{A BBM-like  equation with a nonlocal bilinear term}
Here, we consider a BBM-type equation with  the drift term, but with a nonlocal bilinear term given by a convolution, namely
\be
\label{V10}
u_t-u_{txx} +u_x +  \lambda (u*u)_x=0, \qquad x\in \R,
\ee
where $\lambda\in \R$ is a constant and
\[
(u*v)(x)= \int_{-\infty}^\infty u(x-y)v(y) dy\qquad \text{ for } x\in \R.
\]
A UCP can be derived without any restriction on the initial data.
\begin{theorem}
\label{thm5}
Assume that $\lambda  \ne 0$. Let $u\in C^1([0,T]; H^1( \R ))$ be a solution of \eqref{V10} such that
\be
u(x,t) = 0 \qquad \text{ for } |x|> L,\ t\in (0,T). \label{V11}
\ee
Then $u\equiv 0$.
\end{theorem}
\begin{proof}
Taking the Fourier transform of each term in \eqref{V10} yields
\be
(1+\xi ^2 ) \hat u_t = -i\xi (\hat u + \lambda  \hat{u} ^2 ), \qquad \xi \in\R,\  t\in (0,T).\label{V12}
\ee
Note that, for each $t\in (0,T)$, $\hat u(.,t)$ and $\hat u_t(.,t)$ may be extended to $\C$ as  entire functions of exponential type at most $L$.
Furthermore, \eqref{V12} is still true for $\xi\in\C$ and $t\in (0,T)$ by analytic continuation.
To prove that $u\equiv 0$, it is sufficient to check that
\be
\partial _\xi ^k \hat u (i,t) = 0 \qquad \forall k\in \N,\ \forall t\in (0,T).
\label{V13}
\ee
Let us prove \eqref{V13} by induction on $k$. First, we see that \eqref{V12} gives that either
\be
\hat u (i,t) =0 \qquad \forall t \in (0,T), \label{V14}
\ee
or
\be
\hat u (i,t) =- \lambda ^{-1}\qquad \forall t\in (0,T). \label{V15}
\ee
Derivating with respect to $\xi$ in \eqref{V12} yields (the upperscript denoting the order of derivation in $\xi$)
\be
2\xi \hat u_t (\xi, t) + (1+\xi ^2) \hat u_t ^{(1)} (\xi , t)  = -i \hat u (\xi, t) (1+  \lambda \hat u (\xi, t))  -i\xi \hat u ^{(1)} (\xi , t ) (1 + 2 \lambda \hat u (\xi , t )).
\label{V150}
\ee
Note that $\hat u_t(i,t)=0$ if either \eqref{V14} or \eqref{V15} hold. Combined with \eqref{V150}, this gives
$$
{\hat u} ^{(1)} (i,t)=0,\qquad t\in (0,T).
$$
Assume now that, for some $k\ge 2$,
\be
 \hat u^{(l)}(i,t)=0 \text{ for } t\in (0,T)\text{  and any } l\in \{1,...,k-1 \} .\label{V16}
\ee
Derivating $k$ times with respect to $\xi$ in \eqref{V12} yields
\begin{multline}
(1+\xi ^2) \hat u _t^{(k)} + 2k \xi \hat u _t ^{(k-1)}   + k(k-1)\hat u_t^{(k-2)}= -i\xi \big(\hat u ^{(k)} + \lambda \sum _{l=0} ^k C_k^l \hat u ^{(l)}
\hat u ^{(k-l)}\big)   \\
 -ik \big( \hat u ^{(k-1)}   + \lambda \sum_{l=0}^{k-1} C_{k-1}^l \hat u ^{(l)}\hat u ^{(k-1-l)} \big) . \label{V17}
\end{multline}
From \eqref{V16} and \eqref{V17} we infer that
\[
\hat u ^{(k)}  (i,t) (1+2 \lambda \hat u (i,t)) =0.
\]
Combined to \eqref{V14} and \eqref{V15}, this yields
\[
\hat u ^{(k)} (i,t)=0.
\]
Thus
\be
\hat u ^{(k)} (i,t) =0\qquad \forall k\ge 1. \label{V20}
\ee
\eqref{V15} and \eqref{V20} would imply
\[
\hat u (\xi , t ) =-\lambda  ^{-1} \qquad \forall \xi\in\C ,
\]
which contradicts the fact that $\hat u (.,t)\in L^2(\R )$.  Thus \eqref{V14} holds and $u\equiv 0$.
\end{proof}
%%%%%%%%%%%%%%%%%%%%%%%%%%%%%%%%%%%%%%%%%%%%%%%%%%%%%%%%%%%%%%%%%%
\section{Unique continuation property for the KdV-BBM equation}
\label{sec:UCPKdVBBM}
In this section we prove some UCP for the following  KdV-BBM equation
\be
u_t - u_{txx} -cu_{xxx} + qu_x=0,\qquad x\in \T,\ t\in (0,T), \label{E1}
\ee
where $q\in L^\infty (0,T;L^\infty( \T ))$ is a given potential function and $c\ne 0$ is a given real constant. The UCP obtained here will be used in the next section
to obtain a semiglobal exponential stabilization result for BBM with a moving damping.
\begin{theorem}
\label{thm20}
Let $c\in \R \setminus \{ 0\}$, $T>2\pi /|c|$, and $q\in L^\infty (0,T;L^\infty (\T ))$. Let $\omega \subset \T$ be a nonempty open set.
Assume that $u\in L^2(0,T; H^2(\T ))$ satisfies \eqref{E1} and
\be
u(x,t) =0 \qquad \text{for a.e. } (x,t) \in \omega \times (0,T). \label{E2}
\ee
Then $u\equiv 0$.
\end{theorem}
\begin{proof}
Let $w=u-u_{xx}\in L^2(0,T;L^2(\T ))$. Then $(u,w)$ solves the following system
\ba
&&u-u_{xx} = w, \label{E3}\\
&&w_t + cw_x = (c-q)u_x. \label{E4}
\ea
Note that, by \eqref{E2},
\be
\label{E20}
u=w=0 \qquad \text{ a.e. on } \omega \times (0,T).
\ee
Inspired in part by \cite{AT} (which was concerned with a heat-wave system\footnote{See also \cite{EGGP} for some Carleman estimates for a coupled system of parabolic-hyperbolic equations.}),
we shall establish some Carleman estimates for the elliptic equation \eqref{E3} and the transport equation \eqref{E4} with the {\em same singular weight}.
Introduce a few notations. We identify $\T$ with $(0,2\pi )$. Without loss of generality, we can assume that $c>0$ (the case $c<0$ being similar), and that
$\omega =(2\pi -2\eta, 2\pi + \eta)$ for some $\eta >0$. Let $\omega _0 = (2\pi -\eta , 2\pi )\subset \omega$.
Pick a time $T>2\pi / c$, and some positive numbers $\delta$, $\varepsilon$ and $\rho <1$ such that
 %$\delta \gtrsim 0$, $\varepsilon \gtrsim 0 $, and $\rho \lesssim 1$ such that
\be
\label{abcd}
\rho Tc -2\rho \delta  c - 2 \pi    + \eta -\varepsilon >0.
\ee
Pick finally a function  $g\in C^\infty (0,T)$ such that
\[
g(t) = \left\{
\begin{array}{ll}
\frac{1}{t} \qquad &\text{for } 0<t<\delta /2 , \\
\text{\rm strictly decreasing}\qquad &\text{for }  0<t \le  \delta ,\\
1 \qquad &\text{for } \delta \le t <T.
\end{array}
\right.
\]
Let $\psi \in C^\infty (\T \times [0,T])$ (i.e. $\psi$ is $C^\infty$ smooth in $(x,t)$ and $\psi (.,t)$ is $2\pi$-periodic in $x$ for all $t\in [0,T]$)
with
\[
\psi (x,t) = (x+\varepsilon ) ^2 - \rho c^2 (t-2\delta )^2\qquad \text{ for } x\in [0,2\pi -\eta ],\ t\in [0,T].
\]
%For $ >0$ large enough, let
Let finally
\[
\begin{array}{rll}
\varphi (x,t) &= g(t) (2e^{||\psi ||_{L^\infty}} - e^{\psi (x,t)}), \qquad &(x,t)\in \T\times (0,T],\\
\theta (x,t) &= g(t) e^{\psi (x,t)},\qquad &(x,t)\in \T \times (0,T],
\end{array}
\]
where $||\psi ||_{L^\infty}= ||\psi ||_{L^\infty (\T \times (0,T)) }$.
%Finally, let
%\[ \theta (x,t) = g(t) e^{ \psi (x,t)}. \]
The proof of Theorem \ref{thm20} is outlined as follows. In the first step, we prove a Carleman estimate for the elliptic equation
\eqref{E3} with the time-varying weight $\varphi$. In the second step, we prove a Carleman estimate for the transport equation \eqref{E4}
with the same weight. In the last step, we combine the two above Carleman estimates into a single one for \eqref{E1} and derive the UCP. \\[5mm]
{\sc Step 1. Carleman estimate for the elliptic equation.}
\begin{lemma}
\label{lem1}
There exist $s_0\ge 1$ and $C_0>0$ such that for all $s\ge s_0$ and all
$u\in L^2(0,T;H^2(\T ))$, the following  holds
\begin{multline}
\int_0^T\!\!\!\int_{\T} [  (s\theta ) |u_x|^2 + (s\theta )^3 |u|^2    ]e^{-2s\varphi} dxdt \\
 \le C_0 \left( \int_0^T\!\!\!\int_{\T} |u_{xx}|^2 e^{-2s\varphi} dxdt + \int_0^T \!\!\! \int_{\omega}  (s\theta )^3 |u|^2 e^{-2s\varphi} dxdt \right).  \label{E100}
\end{multline}
\end{lemma}
\noindent
\begin{remark}
The same Carleman estimate as above with terms integrated over $\T$ only is also valid, but with some constants $C_0$ and $s_0$ that could
a priori depend on $t$. The above formulation was preferred for the sake of clarity.
\end{remark}
\noindent
{\em Proof of Lemma \ref{lem1}:} Let $v=e^{-s\varphi}u$ and $P=\partial_x^2$. Then
\[
e^{-s\varphi} Pu = e^{-s\varphi } P(e^{s\varphi } v ) = P_s v + P_av
\]
where
\begin{eqnarray}
P_s v &=&   (s\varphi _x)^2  v + v_{xx}, \\
P_a v &=&  2s \varphi _x v_x + s \varphi _{xx} v
\end{eqnarray}
denote the (formal) selfadjoint and skeweadjoint parts of $e^{-s\varphi } P(e^{s\varphi} \cdot )$.
It follows that
\[
||e^{-s\varphi} Pu||^2 = ||P_s v||^2 + ||P_a v||^2 +  2(P_s v,P_a v)
\]
where $(f,g)=\int_0^T\!\!\!\int_{\T} fg dxdt$, $|| f ||^2 = (f,f)$. In the sequel, $\int_0^T\!\!\!\int_{\T}f(x,t) dxdt $ is denoted $\int\!\!\!\int f $, for the sake of shortness.
Then
\begin{multline*}
(P_s v,P_a v) = \big(  (s\varphi _x)^2 v, 2s\varphi _x v_x\big) + \big(  (s\varphi _x)^2 v, s \varphi _{xx} v \big) \\
+\big( v_{xx}, 2s \varphi _x v_x \big) + (v_{xx}, s\varphi _{xx} v) = : I_1 + I_2 + I_3 + I_4.
\end{multline*}
After some integrations by parts in $x$, we obtain that
\begin{eqnarray*}
&&I_1 =  -3  \ii (s\varphi _x)^2 s\varphi _{xx} v^2 \\
&&I_3 = - \ii s\varphi _{xx} v_x^2 \\
&& I_4 = -\ii v_x (s\varphi _{xxx} v + s\varphi _{xx}v_x) =\ii s\varphi _{xxxx}\frac{v^2}{2} - \ii s\varphi _{xx} v_x^2.
\end{eqnarray*}
Therefore
\[
||e^{-s\varphi } Pu||^2 = ||P_s v|| ^2 + ||P_a v||^2 + \ii [-4(s\varphi _x )^2 s\varphi _{xx} +s\varphi _{xxxx}]v^2 +\ii (-4s\varphi_{xx})v_x^2.
\]
We notice that
\[
\varphi _x = -g \psi _x e^{\psi} , \qquad \varphi _{xx}= -g [(\psi _x)^2 +\psi _{xx}]e^{ \psi},
\]
hence there exist some numbers $s_0\ge 1$, $C>0$ and $C'>0$ such that for all $s\ge s_0$
\[
\begin{array}{rll}
-4(s\varphi _x)^2s\varphi _{xx} + s\varphi _{xxxx}  &\ge C   (sg)  ^3\qquad &\text{ for } (x,t) \in (0,2\pi -\eta )\times (0,T), \\
-4s \varphi _{xx} &\ge  Csg\qquad &\text{ for } (x,t) \in (0,2\pi -\eta )\times (0,T),
\end{array}
\]
while
\[
\begin{array}{rll}
|-4(s\varphi _x)^2s\varphi _{xx} + s\varphi _{xxxx}|  &\le C'   (sg)  ^3\qquad &\text{ for } (x,t) \in (2\pi -\eta , 2 \pi)\times (0,T), \\
|4s \varphi _{xx}| &\le  C'sg\qquad &\text{ for } (x,t) \in (2\pi -\eta , 2\pi )\times (0,T).
\end{array}
\]
We conclude that for $s\ge s_0$ and some constant $C_0>0$
\be
||P_s v||^2 +  \ii [  sg  |v_x|^2 +   (sg)^3|v|^2]
\le C_0 \left(
||e^{-s\varphi }Pu||^2 + \int_0^T\!\!\!\int_{\omega _0}  [ sg  |v_x|^2 +  (sg )^3 |v|^2 ] \right). \label{E30}
\ee
Next we show that $\ii (sg)^{-1} |v_{xx}|^2$ is also less than the r.h.s. of \eqref{E30}. We have
\begin{eqnarray*}
\ii (sg)^{-1} |v_{xx}|^2
&\le & \ii (sg)^{-1} |P_s v - (s\varphi _x)^2 v|^2 \\
&\le& 2 \ii  (sg)^{-1} \big( |P_sv|^2 + |s\varphi _x|^4 |v|^2 \big) \\
&\le& C\left( s^{-1} ||P_sv||^2 + \ii  (sg)^3|v|^2  \right).
\end{eqnarray*}
Combined to \eqref{E30}, this gives
\begin{multline}
\ii \{ (s g )^{-1} |v_{xx}|^2 +  (sg) |v_x|^2 + (sg)^3 |v|^2  \} \\
\le C\left( ||e^{-s\varphi} Pu||^2 +\int_0^T\!\!\!\int_{\omega  _0}  (sg )^3 |v|^2 +
\int_0^T\!\!\! \int_{\omega _0} sg |v_x|^2   \right)
\label{E32}
\end{multline}
where $C$ does not depend on $s$  and $v$.
Finally, we show that we can drop the last term in the r.h.s. of \eqref{E32}. Let $\xi\in C^\infty_0(\omega )$ with $0\le \xi \le 1$ and
$\xi (x) = 1$ for $x\in \omega _0$. Then
\begin{eqnarray*}
\int_0^T\!\!\!\int_{\omega _0} g |v_x|^2
&\le& \int_0^T\!\!\! \int_{\omega} g \xi |v_x|^2 \\
&\le& -\int_0^T\!\!\!\int_{\omega } g ( \xi _x v_x + \xi v_{xx})v \\
&\le& \frac{1}{2}\int_0^T\!\!\!\int_{\omega} g \xi _{xx}  v^2
 -\int_0^T\!\!\!\int_\omega g \xi v_{xx} v
\end{eqnarray*}
so that
\be
2\int_0^T\!\!\!\int_{\omega _0}   sg  |v_x|^2
\le 
||\xi _{xx}||_{ L^\infty (\T) } \int_0^T\!\!\!\int_{\omega}  (sg )|v|^2 + \kappa \int_0^T\!\!\!\int_{\omega} (sg)^{-1}|v_{xx}|^2
+\kappa ^{-1} \int_0^T\!\!\!\int_{\omega}   (sg)^3 |v|^2
 \label{E33}
\ee
where $\kappa >0$ is a constant that can be chosen as small as desired.
Combining \eqref{E32} and \eqref{E33} with $\kappa$ small enough gives for $s\ge s_0$ (with a possibly increased value of $s_0$)
and some constant $C$ (that does not depend on $s$ and $v$)
\be
\ii\{ (sg )^{-1}|v_{xx}|^2 +  (sg )|v_x|^2 +  (sg)^3 |v|^2  \}
\le C \left( ||e^{-s\varphi}  Pu||^2  +\int_0^T\!\!\!\int_{\omega  }  (sg )^3 |v|^2  \right). \label{E34}
\ee
Replacing $v$ by $e^{-s\varphi}u$ in \eqref{E34} gives at once \eqref{E100}. The proof of Lemma \ref{lem1} is complete.  \qed
\\

\noindent
{\sc Step 2. Carleman estimate for the transport equation.}\\
The functions $g,\psi,\varphi$ and $\theta$ are the same as those in Lemma \ref{lem1}.
\begin{lemma}
\label{lem2}
There exist $s_1\ge s_0$ and $C_1>0$ such that for all $s\ge s_1$ and
all $w\in H^1(\T \times (0,T))$, the following  holds
\be
\int_0^T\!\!\!\int_{\T} ( s\theta ) |w|^2 e^{-2s\varphi }dxdt
\le C_1\left(
\int_0^T\!\!\!\int_{\T} |w_t + cw_x|^2 e^{-2s\varphi} dxdt + \int_0^T\!\!\!\int_{\omega} ( s\theta )^2  |w|^2 %e^{- \psi}
e^{-2s\varphi} dxdt
\right). \label{E50}
\ee
\end{lemma}
\noindent
{\em Proof of Lemma \ref{lem2}:} The proof is divided into two parts corresponding to the estimates for $t\in [0,\delta ]$ and  for $t\in [\delta ,T ]$. The main result
in each part is stated in a claim.  Let $v=e^{-s\varphi} w$ and $P=\partial _t + c\partial _x$. Then
\begin{eqnarray*}
e^{-s\varphi } Pw &=& e^{-s\varphi} P (e^{s\varphi } v) \\
&=&(s\varphi _t v + cs\varphi _x v) + (v_t + c v_x) \\
&=:& P_s v + P_a v.
\end{eqnarray*}
\noindent
{\sc Claim 4.}
\begin{multline}
\int_0^\delta \!\!\! \int_{\T} ( s\theta ) ^2  |v|^2  dxdt \\
\le
C\left( \int_0^\delta\!\!\! \int_{\T} %e^{ \psi}
| e^{-s\varphi} Pw|^2 dxdt +
\int_{\T} (1-\xi ) ^2 ( s\theta ) |v|^2_{\vert t=\delta }dx + \int_0^\delta \!\!\!\int_{\omega} ( s\theta ) ^2 %e^{ \psi}
 |v|^2 dxdt   \right). \label{F0}
\end{multline}
To prove the claim, we compute in several ways
\[
I=\int_0^\delta \!\!\! \int_{\T} (e^{-s\varphi} Pw)(1-\xi )^2  s\theta v \, dxdt.
\]
We split $I$ into
\[
I = \int_0^\delta \!\!\!\int_{\T} (P_sv) (1-\xi )^2  s \theta v \, dxdt + \int_0^\delta \!\!\! \int_{\T} (P_a v) (1-\xi )^2  s\theta v \, dxdt =: I_1 + I_2.
\]
Then
\begin{eqnarray*}
I_1 &=& \int_0^\delta \!\!\! \int_{\T}  (\varphi _t + c \varphi _x)(1-\xi )^2  s^2 \theta v^2 \, dxdt \\
&=& \int_0^\delta \!\!\! \int_{\T} [g' (2 e^{ ||\psi||_{L^\infty}} -e^{ \psi}) -g  (\psi _t + c\psi _x) e^{ \psi} ](1-\xi)^2  s^2 g e^{ \psi } v^2\, dxdt.
\end{eqnarray*}
On the other hand
\begin{eqnarray*}
I_2 &=& \int_0^\delta\!\!\! \int_{\T} (v_t+cv_x)(1-\xi ) ^2 ( s g e^{ \psi} v) \, dxdt\\
&=& -\int_0^\delta \!\!\! \int_{\T}  s [g'e^{ \psi} + g  (\psi _t +c\psi _x) e^{ \psi}] (1-\xi )^2 \frac{v^2}{2} dxdt \\
&&\qquad + \frac{1}{2} \int_{\T} (1-\xi) ^2  sg e^{ \psi} |v|^2 _{\vert t=\delta} dx  +
\int_0^\delta \!\!\! \int_{\T} c  s \xi _x (1-\xi ) g e^{ \psi}v^2 dxdt
\end{eqnarray*}
where we used the fact that $e^{-s\varphi} = O(e^{-C/t})$ as $t\to 0^+$ for some constant $C>0$.
Note that for $x\in [0, 2\pi -\eta ] $ and $t\in (0,\delta )$
\[
\psi _t + c\psi _x = 2c (x+\varepsilon ) -2\rho c^2 (t-2\delta ) >2c (\varepsilon + \rho c\delta ) >0
\]
while
\[
g'(t)\le 0\  \text{ and } \ g(t) \ge 1.
\]
Thus, for $s\ge s_1\ge s_0$,
\begin{eqnarray*}
g (\psi _t + c\psi _x) e^{ \psi} ( s^2 g e^{ \psi }+ \frac{ s }{2})  \ge 2c (\varepsilon + \rho c\delta )
 ( sg)^2 e^{2 \psi} ,\quad &&x\in \T\setminus \omega , \ t\in (0,\delta )\\
-g'(t)\left( (2e^{ \Vert \psi\Vert _{L^\infty}}  - e^{ \psi } )  s^2 g e^{ \psi} -\frac{ s}{2} e^{ \psi}  \right) \ge 0\quad &&x\in \T, \ t\in (0,\delta ).
\end{eqnarray*}
It follows that for some positive constants $C,C'$
\be
C\int_0^\delta\!\!\!\int_{\T} ( s\theta )^2  |v|^2dxdt \le -I  + \frac{1}{2}\int_{\T} (1- \xi )^2   sg e^{ \psi} |v|^2 _{\vert t=\delta} dx
+C'  \int_0^\delta\!\!\! \int_{\omega} ( s\theta ) ^2  |v|^2dxdt. \label{F1}
\ee
On the other hand, by Cauchy-Schwarz inequality, we have for any $\kappa >0$
\be
\vert I \vert  \le \kappa ^{-1} \int_0^\delta \!\!\! \int_{\T}  |e^{-s\varphi } Pw|^2dxdt + \kappa \int_0^\delta \!\!\! \int_{\T}  ( s \theta  )^2  |v|^2dxdt. \label{F2}
\ee
Combining \eqref{F1} with \eqref{F2}  gives \eqref{F0} for $\kappa >0$ small enough. Claim 4 is proved. \\
{\sc Claim 5.}
\begin{multline}
\int_\delta^T\!\!\! \int_{\T} ( s\theta ) |v|^2dxdt + \int_{\T} (1-\xi)^2( s\theta) |v|^2_{\vert t=\delta} dx
+ \int_{\T} (1-\xi)^2( s\theta) |v|^2_{\vert t=T} dx \\
\le C \left( \int_\delta ^T \!\!\! \int_{\T} |e^{-s\varphi} Pw|^2 dxdt + \int_\delta ^T \!\!\! \int_{\omega} ( s\theta ) |v|^2  dxdt  \right).
\end{multline}
$\Vert\cdot\Vert$ and $(.,.)$ denoting here the Euclidean norm and scalar product in $L^2(\T \times (\delta ,T))$, we have that
\be
\label{XYZ0}
||e^{-s\varphi} Pw||^2 \ge ||P_s v + P_a v ||^2 \ge ||(1-\xi )(P_s v + P_a v)||^2 \ge 2 ( (1-\xi ) P_s v, (1-\xi )P_a v).
\ee
Next we compute
\begin{eqnarray}
((1-\xi)P_sv,(1-\xi)P_av) &=& \int_\delta ^T \!\!\! \int_{\T} (1-\xi )^2 s (\varphi _t + c\varphi _x) v(v_t + cv_x)  \, dxdt     \nonumber\\
&=&
-\frac{s}{2}\int_\delta^T\!\!\! \int_{\T} (1-\xi )^2(\varphi _{tt} + 2c \varphi _{xt} + c^2 \varphi _{xx})v^2 dxdt  \nonumber \\
 &&
 + \int_{\T} (1-\xi )^2 s(\varphi _t + c\varphi _x) \frac{v^2}{2}dx \bigg\vert _{\delta}^T
 +\int_\delta^T\!\!\! \int_{\T} c\xi _x (1-\xi ) s(\varphi _t + c\varphi _x) v^2 dxdt.\qquad \qquad  \label{XYZ1}
\end{eqnarray}
Recall that $\xi \in C_0^\infty (\omega )$  with $0\le \xi \le 1$ and $\xi (x)=1$ for $x\in \omega _0$, and that $g(t)=1$ for $\delta \le t\le T$, so that
\[
\varphi (x,t) =2 e^{ ||\psi ||_{L^\infty} } -e^{ \psi (x,t) }\quad \text{ for } x\in \T,\  t \in [\delta , T].
\]
We have that
\[
\varphi _t + c\varphi _x = - (\psi _t + c\psi _x) e^{ \psi} = -2 c (x+\varepsilon -\rho c (t-2\delta ))e^{ \psi}
\quad \text{ for } x\in [0, 2\pi -\eta ], \  t\in [\delta ,T].
\]
For $t=\delta $
\[
-s(\varphi _t + c\varphi _ x)(x,\delta) > 2c ( \varepsilon +\rho c\delta ) s e^{ \psi} >0 \qquad \text{for } x\in (0, 2\pi - \eta ),
\]
while for $t=T$, by \eqref{abcd},
\[
s(\varphi _t + c\varphi _ x)(x, T ) > 2c ( \rho Tc -2\rho \delta c -2\pi + \eta  -\varepsilon  ) s e^{ \psi} >0 \qquad \text{for } x\in (0, 2\pi - \eta ).
\]
Therefore
\be
\int_{\T} (1-\xi )^2 s(\varphi _t + c\varphi _x) \frac{v^2}{2} \bigg\vert _\delta ^T \ge
C\left(
\int_{\T}(1-\xi) ^2  s\theta |v|^2 _{\vert t=\delta} dx + \int_{\T} (1-\xi )^2  s\theta |v|^2 _{\vert t=T}dx
\right) .
\label{XYZ2}
\ee
Next we compute
\begin{eqnarray*}
\varphi_{tt}+2c\varphi_{xt} + c^2\varphi_{xx} &=& -\{ (\psi _t + c\psi _x) ^2  +(\psi _{tt} +c^2\psi _{xx})    \} e^{ \psi} \\
&\le& 2(\rho -1 ) c^2 e^{ \psi} \qquad \text{ for } x\in (0,2\pi -\eta ),
\end{eqnarray*}
which yields
\be
-\frac{s}{2} \int_\delta^T \!\!\! \int_{\T}(1-\xi )^2 (\varphi _{tt} +2c\varphi _{xt} +c^2\varphi _{xx})|v|^2dxdt  \ge |1-\rho | c^2 \int_\delta^T\!\!\! \int_{\T} (1-\xi ) ^2  s\theta |v|^2 dxdt.
\label{XYZ3}
\ee
Claim 5 follows from \eqref{XYZ0}-\eqref{XYZ3}.

We infer from Claim 4 and Claim 5 that for some constants $s_1\ge s_0$ and $C_1>0$ we have
for all $s\ge s_1$
\be
\int_0^T\!\!\! \int_{\T}    (s\theta  )|v|^2dxdt \le
 C_1\left(
 \int_0^T\!\!\! \int_{\T}  |e^{-s\varphi} Pw|^2 dxdt+ \int_0^T \!\!\! \int_\omega ( s\theta )^2 |v|^2 dxdt
 \right) . \label{F200}
\ee
Replacing $v$ by $e^{-s\varphi}w$ in \eqref{F200} gives at once \eqref{E50}. The proof of Lemma \ref{lem2} is complete. \\[5mm]
{\sc Step 3.} We would like to apply Lemma \ref{lem1} to $u$ and Lemma \ref{lem2} to $w=u-u_{xx}\in L^2(0,T;L^2(\T))$, which has not the
regularity required. Note, however,  that \eqref{E50} is still true when $w$ and $f:=w_t+cw_x$ are in $L^2(0,T;L^2(\T))$. Indeed,
in that case $w\in C([0,T]; L^2(\T ))$, and if  $(w_0^n)$ and $(f^n)$ are two sequences in $H^1(\T )$ and $L^2(0,T; H^1(\T ))$ respectively, such that
\begin{eqnarray*}
w_0^n &\to& w(0) \quad \text{ in } L^2(\T ) ,\\
f^n &\to& f \qquad\ \, \text{ in } L^2(0,T;L^2(\T)),
\end{eqnarray*}
then the solution $w^n\in C([0,T]; H^1(\T ))$ of
\begin{eqnarray*}
w^n_t + c w^n_x &=&f^n,\\
w^n(0) &=& w_0^n
\end{eqnarray*}
satisfies $w^n \in H^1(\T \times (0,T))$ and $w^n\to w$ in $C([0,T];L^2(\T ))$, so that we can apply
\eqref{E50} to $w^n$ and next pass to the limit $n\to \infty$ in \eqref{E50}.

Here, $u\in L^2(0,T;H^2(\T ))$, $w\in L^2(0,T; L^2(\T ))$ and
$w_t+c w_x = (c-q)u_x \in L^2(0,T;L^2(\T ))$. Combining \eqref{E3}, \eqref{E4}, \eqref{E20}, \eqref{E100}, and \eqref{E50}, we obtain for $s\ge s_1$ that
\begin{multline}
\int_0^T\!\!\! \int_{\T} [  (s\theta ) |u_{x}|^2 + (s\theta ) ^3 |u|^2 +
 (s\theta ) |w|^2] e^{-2s\varphi} dxdt \\
\le C\int_0^T\!\!\!\int_{\T} [|u|^2 +  |w|^2 +|(c-q)u_x|^2 ]e^{-2s\varphi } dxdt.
\end{multline}
We conclude that $u=w=0$ on $\T \times (0,T) $  by choosing $s\ge s_1$ large enough.
\end{proof}

\begin{corollary}
\label{cor20}
The same conclusion as in Theorem \ref{thm20} holds when $u\in L^2(0,T;H^2(\T ))$ is replaced by $u\in L^\infty (0,T; H^1(\T ))$.
\end{corollary}
\begin{proof}
We proceed as in \cite{RZ2006}. Since $u$ and $w:=u-u_{xx}$ are not regular enough to apply Lemmas \ref{lem1} and \ref{lem2},
we smooth them by using some convolution in time. For any function $v=v(x,t)$ and any number $h>0$, we set
\[
v^{[h]}(x,t)=\frac{1}{h}\int _t^{t+h} v(x,s)\, ds.
\]
Recall that if $v\in L^p(0,T;V)$, where $1\le p\le +\infty$ and $V$ denotes any Banach space, then
$v^{[h]}\in W^{1,p}(0,T-h;V)$, $||v^{[h]} ||_{L^p(0,T-h;V)} \le ||v||_{L^p(0,T;V )}$, and for $p<\infty$ and $T'<T$
\[
v^{[h]} \to v \qquad \text{ in } L^p(0,T';V) \quad \text{ as } h\to 0.
\]
%Note that for a.e. $t\in (0,T')$
%\[
%(v^{[h]})_t = \frac{1}{h}\big( v(t+h)-v(t)\big)  = (v_t) ^{[h]}.
%\]
In the sequel, $v^{[h]}_t$ denotes $(v^{[h]})_t$, $v^{[h]}_x$ denotes $(v^{[h]})_x$, etc.
Pick any $T' \in (\frac{2\pi}{|c|},T)$ such that \eqref{abcd} still holds with $T$ replaced by $T'$, and define the functions $g,\psi,\varphi$, 
and $\theta$ as above, but with $T$ replaced by $T'$.  Then for any positive number $h<h_0=T-T'$, $u^{[h]}\in W^{1,\infty} (0,T';H^1(\T ))$, and it solves
\begin{eqnarray}
&&u_t^{[h]}-u_{txx}^{[h]}  -c u_{xxx}^{[h]} + (qu_x)^{[h]} =0\qquad \text{ in }\  {\mathcal D}'(0,T';H^{-2}(\T )),\label{F100}\\
&&u^{[h]}(x,t) =0 \qquad (x,t)\in \omega \times (0,T'). \label{F1000}
\end{eqnarray}
From \eqref{F100}, we infer that
\[
u_{xxx}^{[h]} = c^{-1} \big( u_t^{[h]} -u_{txx}^{[h]} +(qu_x)^{[h]}  \big) \in L^\infty(0,T';H^{-1}(\T )),
\]
hence
\be
\label{F3}
u^{[h]} \in L^\infty (0,T'; H^2(\T )) .
\ee
This yields, with \eqref{E3}-\eqref{E4},
\ba
w^{[h]} &=& u^{[h]} -  u_{xx}^{[h]} \in L^\infty ( 0,T'; L^2(\T)), \label{F4}\\
w_t^{[h]} +cw_x^{[h]} &=& \big( (c-q)u_x \big) ^{[h]} \in W^{1,\infty} (0,T;L^2(\T )). \label{F5}
\ea
From \eqref{F3}, \eqref{F4},  \eqref{F5} and Lemmas \ref{lem1} and \ref{lem2}, we infer that there exist some constants
$s_1>0$ and $C_1 >0$ such that for all $s\ge s_1$ and all $h\in (0,h_0)$, we have
\begin{eqnarray}
&&\int_0^{T'}\!\!\! \int_{\T}
\left(  (s\theta ) |u_x^{[h]}|^2 +  (s\theta )^3 |u^{[h]} |^2
+ (s\theta  ) |w^{[h]}|^2   \right) e^{-2s\varphi} dxdt \nonumber \\
&&\quad \le C_1\int_0^{T'}\!\!\!\int_{\T} \left( |u^{[h]} | ^2 +  |w^{[h]} |^2 + |((c-q)u_x)^{[h]} |^2    \right) e^{-2s\varphi} dxdt  \nonumber\\
&&\quad \le C_1\int_0^{T'}\!\!\!\int_{\T} \left(  |u^{[h]}| ^2 +  |w^{[h]} |^2 + 2|(c-q)u_x^{[h]} |^2    +2|((c-q)u_x)^{[h]} - (c-q)u_x^{[h]} |^2
  \right) e^{-2s\varphi} dxdt. \qquad\qquad \label{F9}
\end{eqnarray}
Comparing the powers of $s$ in \eqref{F9}, we obtain that for $s\ge s_2>s_1$, $h\in (0,h_0)$ and some constant
$C_2>C_1$ (that does not depend on $s,h$)
\begin{eqnarray*}
&&\int_0^{T'}\!\!\! \int_{\T}
\left(  (s\theta ) |u_x^{[h]}|^2 + (s\theta )^3 |u^{[h]} |^2
+ (s\theta )  |w^{[h]}|^2   \right) e^{-2s\varphi} dxdt \\
&&\quad \le C_2\int_0^{T'}\!\!\!\int_{\T}  |((c-q)u_x)^{[h]} - (c-q)u_x^{[h]} |^2  e^{-2s\varphi} dxdt.
\end{eqnarray*}
Fix $s$  to the value $s_2$, and let $h\to 0$. We claim that
\[
\int_0^{T'}\!\!\! \int_{\T} | ( (c-q)u_x)^{[h]} -(c-q)u_x^{[h]} |^2 e^{-2s_2\varphi} \to 0 \qquad \text{ as } h\to 0.
\]
Indeed, as $h\to 0$,
\begin{eqnarray*}
((c-q)u_x)^{[h]} &\to&  (c-q) u_x  \qquad \text{ in } L^2(0,T';L^2(\T)),\\
(c-q)u_x^{[h]} &\to&  (c-q) u_x     \qquad \text{ in } L^2(0,T';L^2(\T )),
\end{eqnarray*}
while $e^{-2s_2\varphi}\le 1$. Therefore,
\[
\int_0^{T'}\!\!\! \int_{\T} \theta ^3 |u^{[h]}|^2 e^{-2s_2\varphi} dxdt \to 0 \qquad \text{ as } h\to 0.
\]
On the other hand, $u^{[h]} \to u $ in $L^2(0,T';L^2(\T))$ and $\theta ^3e^{-2s_2\varphi}$ is bounded on $\T \times (0,T')$, so that
\[
\int_0^{T'}\!\!\! \int_{\T} \theta ^3 |u^{[h]}|^2 e^{-2s_2\varphi} dxdt \to \int_0^{T'}\!\!\! \int_{\T} \theta ^3 |u|^2 e^{-2s_2\varphi} dxdt
\]
as $h\to 0$. We conclude that $u\equiv 0$ in $\T \times (0,T')$. As $T'$ may be taken arbitrarily close to $T$, we infer that
$u\equiv 0$ in $\T\times (0,T)$, as desired.
\end{proof}
%%%%%%%%%%%%%%%%%%%%%%%%%%%%%%%%%%%%%%%%%%%%%%%%%%%%%%%%%%%%%%%%%%
\section{Control and Stabilization of the KdV-BBM equation}
\label{sec:control}
In this section we are concerned with the control properties of the system
\begin{eqnarray}
&&u_t - u_{txx} -cu_{xxx} +(c+1)u_x + uu_x= a(x) h, \qquad x\in \T ,\ t\ge 0, \label{D1}\\
&&u(x,0)= u_0(x), \label{D2}
\end{eqnarray}
where $c\in \R \setminus \{ 0 \} $ and $a\in C^\infty(\T )$ is a given nonnul function. Let
\begin{equation} \label{D3}
\omega = \{ x\in \T ; \ a(x)\ne 0 \} \ne \emptyset .
\end{equation}
\subsection{Exact controllability}
The first result is a local controllability result in large time.
%%%%%%%%%%%%%%%%%%%%%%%%%%%%%%%%%%%%%%%%%%%%%%%%%%%%%%%%%%%%%%%%%%
%%%%%%%%%%%%%%%%%%%%%%%%%%%%%%% THEOREM  1 %%%%%%%%%%%%%%%%%%%%%%%%%%%%
\begin{theorem}
\label{thm1} Let $a\in C^\infty(\T)$ with $a\ne 0$, $s\geq 0$ and
$T>2\pi/|c|$. Then there exists a $\delta >0$ such that for any
$u_0,u_T\in H^s(\T)$ with
\[ ||u_0||_{H^s} + ||u_T||_{H^s} < \delta ,\]  one
can find  a control input $h\in L^2(0,T;H^{s-2}(\T ))$ such that the
system \eqref{D1}-\eqref{D2} admits a unique solution  $u\in
C([0,T],H^s(\T ))$  satisfying $u(\cdot ,T)=u_T$.
\end{theorem}
\begin{proof}
%%%%%%%%%%%%%%%%%%%%%%%%%%%%%%%%%%%%%%%%%%%%%%%%%%%%%%%%%%%%%%%%%%%
The result is first proved for the linearized equation, and next extended to the nonlinear one by a fixed-point argument.\\
\smallskip
{\sc Step 1. Exact controllability of the linearized system}\\
\medskip
We first consider  the exact controllability  of the linearized
system
 \ba
&&u_t-u_{txx} -cu_{xxx} +(c+1)u_x = a(x)h, \label{D5}\\
&&u(x,0)=u_0(x),\label{D6} \ea in $H^s(\T )$ for any $s\in
\mathbb{R} $.   Let $A=( 1-\partial _x^2)^{-1}(c\partial_x^3
-(c+1)\partial _x )$ with domain $D(A)=H^{s+1}(\T )\subset H^s(\T )$. 
The operator $A$ generates a group of isometries $\{ W(t) \} _{t\in \R}$ in $H^s(\T )$, with
\[ W(t)v  =
\sum_{k=-\infty}^{\infty} e^{-it\frac{ck^3+(c+1)k}{k^2+1}} {\hat
v}_k e^{ikx} 
\]
 for  any \[ v= \sum_{k=-\infty}^{\infty} \hat
{v}_ke^{ikx}   \in H^s (\T).\] 
%Let $Bu=au$. Note that $B\in {\mathcal L } (H^s(\T ))$ for all $s\in \R$.  
The system \eqref{D5}-\eqref{D6} may be cast into the
%may be written as
%\ba
%u_t &=& Au + Bh,\label{AB1}\\
%u(x,0)&=&u_0(x). \label{AB2}
%\ea
following integral form
 \[
 u(t) = W(t) u_0 + \int_0^t W(t-\tau) (  1 -\partial _x^2) ^{-1} [a(x) h(\tau)] d\tau.
 \]
We proceed as in \cite{MORZ}. 
Take $h(x,t)$ in  \eqref{D5} to have the following form
 \newcommand{\jsum}{\sum ^{\infty}_{j=-\infty}}
 \begin{equation}\label{h}
  h(x,t) = a(x) \sum ^{\infty}_{ j = -\infty}f_j q_j (t)  e^{ijx}
  \end{equation}
  where $f_j $ and $q_j (t)$ are to be determined later. Then the
  solution $u$ of the equation (\ref{D5}) can be written as
  \[ u(x,t)= \sum ^{\infty}_{k=-\infty}  \hat{u}_k(t) e^{ikx} \]
  with $\hat{u}_{k} (t) $ solves
  \begin{equation}\label{h1}
  \frac{d}{dt} \hat{u}_k (t) + ik\sigma (k) \hat{u}_k (t) =
  \frac{1}{1+k^2} 
  \jsum f_j q_j(t)  m_{j,k}
  \end{equation}
  where $\sigma (k)= \frac{ck^2+c+1}{1+k^2} $, and
  \[ m_{j,k} =\frac{1}{2\pi} \int _{\T} a^2 (x)e^{i(j-k)x} dx .\]
  Thus
  \[ \hat{u}_k (T)-e^{-ik\sigma (k)T} \hat{u}_k(0)=  \frac{1}{1+k^2}
  \jsum f_j m_{j,k} \int ^T_0 e^{-ik\sigma(k)(T-\tau)} q_j (\tau )
  d\tau \]
  or
\[ \hat{u}_k (T)e^{ik\sigma (k)T}- \hat{u}_k(0)= \frac{1}{1+k^2}
  \jsum f_j m_{j,k} \int ^T_0 e^{ik\sigma(k)\tau} q_j (\tau )
  d\tau .\]
It may occur that  the eigenvalues 
  \[ 
  \lambda _k = i k \sigma (k), \ k \in \Z 
  \] 
  are not all different. If we count only the distinct values, we obtain the sequence $(\lambda _k)_{k\in \cI}$, where $\cI \subset \Z$ has the property that 
  $\lambda _{k_1}\ne \lambda _{k_2}$ for any $k_1,k_2\in \cI$ with $k_1\ne k_2$. For each $k_1\in \Z$ set
  \[
  I(k_1)=\{ k\in\Z ;\ k\sigma (k) = k_1\sigma (k_1) \}
  \]
  and $m(k_1)= | I(k_1) |$ (the number of elements in $I(k_1)$). Clearly,
  there exists some integer $k^*$ such that $k\in \cI$ if $| k | >k^*$. 
  Thus there are only finite many integers in $\cI$, say $k_j$, $j=1,...,n$, such that one can find another
  integer $k\ne  k_j$ with $\lambda _k =\lambda _{k_j}$. Let 
  \[
  \cI _j=\{k\in \Z; \ k\ne k_j,\ \lambda _k=\lambda _{k_j}\}, \ j=1,2,...,n.
  \]
  Then 
  \[
  \Z = \cI \cup \cI _1 \cup ... \cup \cI _n.
  \]
  Note that $\cI _j$ contains at most two integers, for $m(k_j) \le 3$. We write 
  \[
  \cI _j = \{ k_{j,1},k_{j,m(k_j)-1} \}\quad j=1,2,...,n
  \]
 and rewrite $k_j$ as $k_{j,0}$.  
 Let
  \[ p_k(t):= e^{-ik\sigma (k) t}, \quad k=0, \pm 1, \pm 2,...
  \]
  Then the set
\[\mathcal{P}:= \{ p_k (t); \  k\in \cI \} \]
  forms a Riesz basis for its closed span, $\mathcal{P}_T,$ in $L^2(0,T)$ if
  \[ T> \frac{2\pi}{|c|} \cdot \]
  Let $\mathcal{L}:= \{ q_j (t); \ j\in \cI  \}$ be
  the unique dual Riesz basis for $\mathcal{P}$ in $\mathcal{P}_T$; that is,
  the functions in $\mathcal{L}$ are the unique elements of $\mathcal{P}_T$
  such that
  \[ \int ^T_0 q_j (t)\overline{p_k (t)} dt = \delta _{kj}, \ j,k\in \cI .\]
  In addition, we choose
  \[
  q_k = q_{k_j} \qquad \text{ if } k\in \cI _j.
  \]
  For such choice of $q_j (t)$, we have then, for any $k\in \Z$,
  \be 
  \label{S1}
  \hat{u}_k (T)e^{ik\sigma (k)T}-\hat{u}_k (0)= \frac{1}{1+k^2}
  f_k m_{k,k} \quad \text{ if } k\in \cI \setminus \{ k_1,...,k_n\}; 
  \ee
  \be 
  \label{S2}
  \hat{u}_{k_{j,q}} (T)e^{ik_j\sigma (k_j)T}-\hat{u}_{k_{j,q}} (0)= \frac{1}{1+k_{j,q}^2}
  \sum_{l=0}^{m(k_j) -1} f_{k_{j,l}} m_{k_{j,l},k_{j,q}} \quad \text{ if } k=k_{j,q}, \ j=1,...,n,\  q=0,...,m(k_j)-1. 
  \ee
  It is well known that for any finite set $\mathcal J\subset \Z$, the  Gram matrix 
  $A_{\mathcal J}=(m_{p,q})_{p,q \in \mathcal J}$ is definite positive, hence invertible.
It follows that the system \eqref{S1}-\eqref{S2} admits a unique solution
 $\vec f (...,f_{-2},f_{-1},f_0,f_1,f_2,...)$. Since
  \[ m_{k,k}  %=\frac{1}{2\pi} \int _{\T} a^2 (x) e^{i(k-k)x}dx 
  = \frac{1}{2\pi} \int _{\T} a^2(x) dx =: \mu \ne 0, \]
  we have that
  \[ f_k = \frac{1+k^2}{\mu} 
 % \mu ^{-1}
  \left(  \hat{u}_k (T )e^{ik\sigma (k)T}-\hat{u}_k
  (0)\right)     \qquad \text{ for } |k| > k^*  \]
  Note that 
  \begin{eqnarray*}
  ||h||^2_{L^2(0,T;H^{s-2}(\T ))} 
  &=& \int_0^T || a(x) \sum_{j=-\infty}^\infty f_jq_j(t)e^{ijx} ||^2_{H^{s-2}}dt\\
  &\le& C\int_0^T \sum_{j=-\infty}^\infty  (1+j^2)^{s-2}|f_jq_j(t)|^2\\
  &\le& C\sum_{j=-\infty}^\infty (1+j^2)^{s-2} |f_j|^2 \\
  &\le& C\left( ||u(0)||^2_{H^s} + ||u(T)||^2_{H^s} \right) .
  \end{eqnarray*}
  
  This analysis leads us to the
  following controllability result for the linear system (\ref{D5})-(\ref{D6}).
  
  \begin{proposition} Let $s\in \R $ and $T>\frac{2\pi}{|c|}$ be
  given. For any $u_0, u_T\in H^s (\T)$, there exists a control 
  $h\in L^2 (0,T; H^{s-2} (\T))$ such that the system (\ref{D5})-(\ref{D6}) admits a unique
  solution $u\in C([0,T]; H^s (\T))$ satisfying
  \[ u(x,T)=u_T (x).\]
  Moreover, there exists a constant $C>0$ depending only on $s$ and
  $T$ such that
  \[ \| h\|_{L^2 (0,T; H^{s-2} (\T))} \leq C \left (\|u_0 \|_{H^s}
  +\|u_T\|_{H^s} \right) .\]
  \end{proposition}

Introduce the (bounded) operator $\Phi : \ H^s(\T )\times H^s(\T )  \to L^2(0,T;H^{s-2}(\T ))$  defined by
\[ \Phi (u_0, u_T)(t) = h(t ), \]
where $h$ is given by \eqref{h} and $\vec f$ is the solution of \eqref{S1}-\eqref{S2} with 
$(\widehat{u_0})_k$ and $(\widehat{u_T})_k$ substitued to $\hat{u}_k(0)$ and 
$\hat{u}_k(T)$, respectively.  

Then $h=\Phi (u_0,u_T)$ is a
control driving the solution $u$ of \eqref{D5}-\eqref{D6}
from $u_0$ at $t=0$ to $u_T$ at $t=T$. \\

\noindent
{\sc Step 2. Local exact controllability of the BBM equation.}\\

We proceed as in \cite{Rosier97}. Pick any time $T>2\pi / |c|$, and
any $u_0,u_T\in H^s (\T )$ ($s\ge 0$) satisfying \[ \| u_0\|_{H^s}\leq \delta,
\quad \|u_T\|_{H^s} \leq \delta \]
with $\delta $ to be determined. For
any $u\in C([0,T];H^s(\T ))$, we set
\[
\omega (u) = - \int_0^T W(T-\tau)(  1 -\partial _x ^2)^{-1}(uu_x)(\tau)\, d\tau.
\]
Then
\[
||\omega (u)-\omega (v) ||_{H^s} \le CT ||u+v||_{L^\infty (0,T;H^s(\T
))} ||u-v||_{L^\infty(0,T;H^s(\T ))}.
\]
Furthermore,
\begin{eqnarray*}
&&W(t) u_0 + \int_0^t W(t-\tau ) (1-\partial _x ^2)^{-1}  [a(x) \Phi (u_0, u_T -\omega (u))  
- uu_x ](\tau )d\tau \\
&&\qquad =
\left\{
\begin{array}{ll}
u_0 \qquad &\text{ if } t=0,\\
\omega ( u ) +(u_T -\omega (u) )=u_T  &\text{ if } t=T.
\end{array}
\right.
\end{eqnarray*}
We are led to consider the nonlinear map
\[
\Gamma (u) = W(t) u_0 +\int_0^t W(t-\tau )  (1-\partial _x^2) ^{-1} [ a(x) \Phi (u_0, u_T  -\omega (u)) - uu_x  ](\tau )\,d\tau .
\]
The proof of Theorem \ref{thm1} will be complete if we can show that
the map $\Gamma $ has a fixed point in some closed ball of the space
$ C([0,T];H^s(\T ))$. For any $R>0$, let
\[
B_R =\{ u\in C([0,T];H^s(\T )); \ ||u||_{C([0,T];H^s(\T ))} \le R\}.
\]
From the above calculations, we see that there exist two positive
constants $C_1,C_2$ (depending on $s$ and $T$, but not on $R$,
$||u_0||_{H^s}$ or $||u_T||_{H^s}$) such that for all $u,v\in B_R$
\begin{eqnarray*}
|| \Gamma (u) ||_{C([0,T];H^s(\T ))} &\le& C_1 \big( ||u_0||_{H^s } + ||u_T||_{H^s } \big) + C_2 R^2,\\
|| \Gamma (u) -\Gamma (v) ||_{C([0,T];H^s(\T ))} &\le & C_2 R||
u-v||_{C([0,T];H^s(\T ))}.
\end{eqnarray*}
Picking $R=(2C_2)^{-1}$ and $\delta = (8C_1C_2)^{-1}$, we obtain for
$u_0,u_T$ satisfying \[ \|u_0\|_{H^s} \leq \delta, \quad \|u_T\|_{H^s} \leq
\delta \]
and $u,v\in B_R$ that
\begin{eqnarray}
||\Gamma (u)||_{C([0,T];H^s(\T ))} &\le & R\\
||\Gamma (u) -\Gamma (v) ||_{C([0,T];H^s(\T ))} & \le & \frac{1}{2}
||u-v||_{C([0,T];H^s(\T ))}.
\end{eqnarray}
It follows from the contraction mapping theorem that $\Gamma$ has a unique fixed point $u$ in $B_R$. Then $u$ satisfies \eqref{D1}-\eqref{D2} 
with $h=\Phi (u_0,u_T -\omega (u))$ and
$u(T)=u_T$, as desired. The proof of Theorem \ref{thm1} is complete.
\end{proof}
\subsection{Exponential stabilizability}
We are now concerned with the stabilization of  \eqref{D1}-\eqref{D2} with a feedback law $h=h(u)$. To guess the expression of $h$, it is convenient
to write the linearized system  \eqref{D5}-\eqref{D6} as
\ba
&&u_t=Au + B k, \label{G1}\\
&&u(0)=u_0 \label{G2} 
\ea 
where $k(t) = ( 1 -\partial _x^2)^{-1} h(t) \in L^2(0,T; H^s(\T ))$ is the new control input, and 
\be
\label{G3} B =( 1-\partial _x^2) ^{-1} a( 1 -\partial _x ^2) \in
{\mathcal L} (H^s(\T )). 
\ee 
We already noticed that $A$ is skew-adjoint in $H^s(\T )$, and that \eqref{G1}-\eqref{G2} is exactly controllable
in $H^s (\T )$ (with some control functions $k\in L^2(0,T;H^s(\T ))$) for any $s\ge 0$. If we choose the simple feedback law 
\be k = -B^{*,s}u, \label{G4} \ee 
the resulting closed-loop system 
\ba
&&u_t=Au - B B^{*,s}u, \label{G5} \\
&&u(0)=u_0  \label{G6} 
\ea 
is exponentially stable in $H^s (\T)$ (see e.g. \cite{Liu,Rosiersurvey}.)
%Note that the system \eqref{G5}-\eqref{G6}, when expanded, reads
%\ba
%&& u_t-u_{txx} -cu_{xxx} +(c+1)u_x = -(1-\partial _x^2)[a^2u], \label{G5bis}\\
%&& u(x,0)=u_0(x), \label{G6bis} 
%\ea
%and that  
%\[
%\text{supp} (1-\partial _x ^2)[a^2u] \subset \overline{\omega}. 
%\]
In \eqref{G4}, $B^{*,s}$
denotes the adjoint of $B$ in ${\mathcal L} (H^s(\T ))$. Easy computations show that 
\be B^{*,s}u= (1-\partial _x^2)^{1-s} a (1-\partial _x^2)^{s-1} u  \label{G7}. \ee 
In particular
\[
B^{*,1} u = au.
\]
Let  $\tilde A = A-B B^{*,1}$, where $(BB^{*,1})u=(1-\partial _x ^2)^{-1}[a(1-\partial _x^2 ) (au)]$. Since $BB^{*,1}\in {\mathcal L} (H^s(\T ))$ and $A$ is
skew-adjoint in $H^s(\T )$, $\tilde A$ is the infinitesimal
generator of a group $\{ W_a(t)\}_{t\in \R }$ on $H^s(\T )$ (see
e.g. \cite[Theorem 1.1 p. 76]{Pazy1983}). We first show that the closed-loop system \eqref{G5}-\eqref{G6} is exponentially stable
in $H^s(\T)$ for all $s\ge 1$.
\begin{lemma}\label{linear} 
Let $a\in C^{\infty} (\T)$ with $a\ne 0$. Then
there exists  a constant $\gamma >0$  such that for any  $s\geq 1$, one can find a constant $C_s>0$ for which
the following holds for all $u_0\in H^s(\T )$ 
\be
\|W_a(t)u_0\|_{H^s} \leq C_s   e^{-\gamma t}   \|u_0\|_{H^s}  \qquad \text{for all } \ t\geq 0.
\label{decayHs}
\ee
%\item[(ii)] for any $T>0$ and  $f\in L^1(0,T; H^s(\T))$, \[u(x,t):= \int ^t_0 W_a (t-\tau) f(\tau) d\tau  \in C([0,T]; H^s (\T))\] satisfying
%\[ \|u\|_{C(0,T]; H^s (\T))} \leq C \| f\|_{L^1 (0,T; H^s (\T))} .\]
%\end{itemize}
\end{lemma}
\begin{proof} \eqref{decayHs}  is well known for $s=1$  (see e.g. \cite{Liu}). Assume that it is true for
some $s\in \N ^*$, and pick any  $u_0\in H^{s+1}(\T )$. 
Let $v_0=\tilde A u_0\in H^s(\T )$. Then 
\[
||W_a (t) v_0||_{H^s} \le C_s e^{-\gamma t}     ||v_0||_{H^s}. 
\]  
Clearly, 
\[
W_a(t) v_0= \tilde A W_a(t) u_0= AW_a (t) u_0 - BB^{*,1}W _a (t) u_0, 
\]
hence
\[
||AW_a (t) u_0||_{H^s} \le ||W _a (t) v_0|| _{H^s} + ||BB^{*,1}||_{ {\mathcal L } (H^s) }  ||W_a (t) u_0||_{H^s}  \le C e^{-\gamma t}  ||u_0||_{H^{s+1}}     \cdot  
\]
Therefore
\[
||W_a (t) u_0||_{H^{s+1}} \le C_{s+1}  e^{-\gamma t}  ||u_0||_{H^{s+1}} ,   
\] 
as desired. The estimate \eqref{decayHs} is thus proved for any $s\in \N ^*$. It may be extended to any $s\in [1,+\infty )$ by interpolation. 
\end{proof}

Plugging the feedback law $k = -B^{*,1}u = -au$ in the nonlinear equation
gives the following closed-loop system \ba
&& u_t-u_{txx} -cu_{xxx} +(c+1)u_x +uu_x = -a (1-\partial _x^2)[au], \label{G8}\\
&& u(x,0)=u_0(x). \label{G9} \ea We first show that  the system
\eqref{G8}-\eqref{G9} is globally well-posed in the space $H^s
(\T)$ for any $s\geq 0$.
\begin{theorem} \label{global} Let $s\geq 0$  and $T>0$ be
given. For any $u_0\in H^s (\T)$, the system \eqref{G8}-\eqref{G9}
admits  a unique solution $u\in C([0,T]; H^s(\T ))$.
\end{theorem}
The following bilinear estimate from \cite{Roumegoux}
will be very helpful.
\begin{lemma}\label{bilinear}
Let  $w\in H^r (\T) $ and $v\in H^{r'}(\T)$ with $0\leq r\leq s$,
$0\leq r'\leq s$ and $\ 0\leq 2s -r-r' < \frac14$.  Then
\[ \|(1-\partial ^2_x)^{-1} \partial _x (wv)\|_{H^s}\leq c_{r,r',s}
\|w\|_{H^r} \|v\|_{H^{r'}} .\] In particular, if $w\in H^r (\T )$ and $v\in H^s(\T )$ with $0\leq r\leq s< r+\frac14$, then
\[ \|(1-\partial ^2_x)^{-1} \partial _x (wv)\|_{H^s}\leq c_{r,s}
\|w\|_{H^r}\|v\|_{H^s}. \] 
\end{lemma}

\noindent {\em Proof of Theorem \ref{global}:}

\smallskip
\noindent
 \emph{Step 1}:    The system is locally
well-posed in the space $H^s (\T)$:

\medskip
 \emph{Let $s\geq 0$ and $R>0$ be given. There exists a $T^*$
depending only on $s$ and $R$  such that for any $u_0\in H^s (\T)$
with 
\[ \|u_0\|_{H^s}  \leq R,\] 
the system \eqref{G8}-\eqref{G9} admits
a unique solution $u\in C([0,T^*]; H^s (\T))$. Moreover,  $T^*\to
\infty $ as $R\to 0$.}

\bigskip
Rewrite \eqref{G8}-\eqref{G9} in its integral form \be u(t) = W_a(t)
u_0 -\int_0^t W_a(t- \tau  ) (1-\partial _x ^2) ^{-1} (uu_x)(\tau )
d\tau . \label{G50} \ee
For given $\theta>0$, define a map  $\Gamma
$ on $C([0, \theta ]; H^s (\T))$ by
\[ \Gamma (v)=  W_a(t)
u_0 -\int_0^t W_a(t- \tau  ) (1-\partial _x ^2) ^{-1} (vv_x)(\tau )
d\tau \] for any $v\in C([0,\theta]; H^s(\T))$. Note that, according
to Lemma \ref{linear} and Lemma \ref{bilinear},
\[ \| W_a(t)u_0\|_{C([0,\theta];H^s(\T))} \leq C_s\| u_0\|_{H^s} , \]
and
\begin{eqnarray*}
\left \| \int_0^t W_a(t- \tau  ) (1-\partial _x ^2) ^{-1}
(vv_x)(\tau ) d\tau \right \| _{C([0, \theta];H^s (\T))} 
&\leq& C_s\theta \sup_{0\leq t\leq \theta}\|
(1-\partial _x ^2) ^{-1} (vv_x)(t )\|_{H^s} \\
 &\leq& \frac{C_s c_{s,s}}{2} \theta \| v\|_{C([0,\theta]; H^s (\T))}^2 .
\end{eqnarray*}
Thus, for given $R>0$ and $u_0\in H^s (\T)$ with $\|u_0\|_{H^s} \leq
R$, one can choose $T^*=[2c_{s,s}(1+C_s)  R]^{-1}$ such that $\Gamma $ is a
contraction mapping in the ball $$B:= \{ v\in C([0, T^*]; H^s(\T));
\ \| v\|_{C([0,T^*]; H^s (\T))} \leq 2C_s R\} $$  whose fixed point
$u$ is the desired solution.

 \medskip
 \noindent
 {\em Step 2}:  The system is globally well-posed in the space $H^s (\T)$ for any $s\geq 1$.

\medskip
 To this end, it suffices to establish the following global \emph{a priori}
 estimate for smooth solutions of the system \eqref{G8}-\eqref{G9}:

 \medskip
\emph{Let $s\geq 1$ and $T>0$ be given. There exists a continuous
 nondecreasing
 function
 \[ \alpha _{s,T} : \R^+\to \R^+\]
 such that
any smooth solution $u$ of the system \eqref{G8}-\eqref{G9}
 satisfies
 \begin{equation}\label{priori}
 \sup _{0\leq t\leq T}\|u(\cdot, t)\|_{H^s} \leq \alpha _{s,T}
 (\|u_0\|_{H^s}).
 \end{equation}
}

 \medskip
 Estimate \eqref{priori} holds obviously  when $s=1$ because of the
  energy identity
\[
||u(t)||^2_{H^1}  - ||u_0||^2_{H^1}  = - 2 \int_0^t ||a u(\tau )||^2_{H^1}
d\tau  \qquad  \forall   t\geq 0.
\]
When $1<s\leq s_1:=1+\frac{1}{8}$, applying Lemma \ref{linear} and
Lemma \ref{bilinear} to (\ref{G50}) yields that for any $0< t\leq T$,
\begin{eqnarray*}
 \|u(\cdot, t)\|_{H^s} 
&\leq&  C_s\|u_0\| _{H^s}  +\frac{C_s c_{1,s}}{2} \int ^t_0 \| u(\cdot , \tau)\|_{H^1} \|u(\cdot, \tau )\|_{H^s} d \tau \\
&\leq& C \|u_0\| _{H^s} + C \alpha _{1,T} (\| u_0\|_{H^1}) \int ^t_0 \| u(\cdot, \tau )\|_{H^s} d\tau . 
\end{eqnarray*}
Estimate (\ref{priori}) for $1<s\leq s_1$ follows by using
Gronwall's lemma.  Similarly, for $s_1 < s\leq s_2 :=
1+\frac{2}{8}$,
\begin{eqnarray*}
\| u(\cdot, t)\|_{H^s} 
&\leq&  C_s \|u_0\| _{H^s}  +\frac{C_s c_{s_1,s} }{2}
 \int ^t_0 \|  u(\cdot , \tau)\|_{H^{s_1}} \| u(\cdot, \tau )\|_{H^s} d \tau \\
&\leq&  C \|u_0\| _{H^s} + C
\alpha _{s_1,T} (\| u_0\|_{H^{s_1}} ) \int ^t_0 \| u(\cdot, \tau )\|_{H^s}
d\tau .
\end{eqnarray*}
Estimate (\ref{priori}) thus holds for $1<s\leq s_2$. Continuing
this argument, we can show that the estimate (\ref{priori}) holds
for $1<s\leq s_k:=1+\frac{k}{8}$ for any $k\geq 1$.

\medskip
\noindent {\em Step 3:} \emph{The system \eqref{G8}-\eqref{G9} is
globally well-posed in the space $H^s (\T)$ for any $0\leq s<1$}.

\medskip
To see it is true, as in \cite{Roumegoux}, we decompose  any $u_0\in H^s (\T)$ as
\[ u_0 =\sum_{k\in \Z } \hat u_k e^{ikx}= \sum_{|k| \le k_0} + \sum_{|k| > k_0} =: w_0 +v_0 \] 
with $ v_0\in H^s (\T)$ satisfying \[
\|v_0\| _{H^s} \leq \delta \] for some small $\delta >0$ to be chosen, and
$w_0\in H^1(\T)$. Consider the following two initial value problems
\begin{equation}\label{m-1}
\begin{cases}
 v_t-v_{txx} -cv_{xxx} +(c+1)v_x +vv_x = -a (1-\partial _x^2)[av], \\
v(x,0)=v_0(x)
\end{cases}
\end{equation}
and
\begin{equation}\label{m-2}
\begin{cases}
 w_t-w_{txx} -cw_{xxx} +(c+1)w_x +ww_x  +(vw)_x= -a (1-\partial _x^2)[aw], \\
w(x,0)=w_0(x).
\end{cases}
\end{equation}
By the local well-posedness established in Step 1,  for given $T>0$,
if $\delta $ is small enough, then (\ref{m-1}) admits a unique
solution $v\in C([0,T]; H^s (\T))$. For (\ref{m-2}), with $v\in
C([0,T]; H^s (\T))$, by using Lemma \ref{linear}, the estimate
\[
||(1-\partial _x ^2)^{-1} \partial _x (wv) ||_{H^1} \le C|| w v ||_{L^2} \le C ||w||_{H^1}  ||v||_{H^s}
\]
and the contraction mapping principle, one can show
first  that it is locally well-posed in the space $H^1(\T)$.  Then,
for any smooth solution $w$ of (\ref{m-2}) it holds that
\[ \frac12 \frac{d}{dt} \| w(\cdot, t)\|^2_{H^1} -\int_{\T} 
v(x,t)w(x,t)w_x(x,t) dx = -  \| a(\cdot) w (\cdot, t)\|_{H^1}^2  , \]
which implies that
\[ \| w(\cdot, t)\| ^2_{H^1}\leq \|w_0\| ^2_{H^1} \exp \left (C\int ^t_0 \|
v(\cdot, \tau ) \|_{L^2 } d\tau  \right ) \] for any $t\geq 0$.
The above estimate can be extended to any $w_0\in H^1(\T)$ by a density argument.
Consequently, for $w_0\in H^1 (\T)$ and $v\in C([0,T]; H^s(\T))$,
(\ref{m-2}) admits a unique solution $w\in C([0,T]; H^1 (\T))$.
Thus $u=w+v\in C([0,T]; H^s (\T))$ is the desired solution of system
\eqref{G8}-\eqref{G9}. The proof of Theorem 6.4 is complete.\qed

\medskip
Next we show that the system \eqref{G8}-\eqref{G9} is locally
exponentially stable in $H^s (\T)$ for any $s\geq 1$.
\begin{proposition}\label{locale} 
Let $s\geq 1$ be given and $\gamma >0$ be as given in Lemma \ref{linear}. Then there exist two 
numbers $\delta >0$ and $C'_s$
depending only on $s$ such that for any
$u_0\in H^s(\T)$ with \[ \|u_0\|_{H^s} \leq \delta ,\] the corresponding
solution $u$ of the system \eqref{G8}-\eqref{G9} satisfies \[
\|u(\cdot, t) \|_{H^s} \leq C'_s   e^{-\gamma t}  \|u_0\|_{H^s}  \quad \forall 
t\ge 0.\]
\end{proposition}
\begin{proof} We proceed as in \cite{PR}. 
As in the proof of Theorem \ref{global}, rewrite the
system \eqref{G8}-\eqref{G9} in its integral form
\[ 
u(t)=W_a (t) u_0 -\frac12 \int ^t_0 W_a (t-\tau) ( 1 -\partial
^2_x)^{-1}\partial _x (u^2)(\tau ) d\tau \] 
and consider the map
\[ \Gamma (v):= W_a (t) u_0 -\frac12 \int ^t_0 W_a (t-\tau) (1-\partial
^2_x)^{-1}\partial _x (v^2)(\tau ) d\tau .\] 
For given $s\geq 1$, by
Lemma \ref{linear} and Lemma \ref{bilinear}, there exists a constant $C_s>0$ such that
\begin{eqnarray*}
\| \Gamma (v) (\cdot, t)\|_{H^s} & \leq& C_s e^{-\gamma t} \|u_0\|_{H^s}
+\frac{C_sc_{s,s}}{2} \int ^t_0 e^{-\gamma (t-\tau)} \| v(\cdot, \tau )\|_{H^s}^2 d\tau \\
& \leq& C_s e^{-\gamma t} \|u_0\|_{H^s}
+\frac{C_sc_{s,s}}{2}  \sup _{0\leq \tau \leq t} \|e^{\gamma \tau } v(\cdot, \tau )\|_{H^s}^2 \int ^t_0 e^{-\gamma (t+\tau)} d\tau \\
&\leq& C_s e^{-\gamma t} \|u_0\|_{H^s} + \frac{C_sc_{s,s}}{2\gamma}e^{-\gamma
t} (1-e^{-\gamma t}) \sup _{0\leq \tau \leq t} \|e^{\gamma \tau}
v(\cdot, \tau )\|_{H^s}^2 \end{eqnarray*} 
for any $t\geq 0$. Let us introduce the Banach space
\[
Y_s:=\{ v\in C([0, \infty ); H^s (\T)): \ || v ||_{Y_s}:=
\sup_{0\leq t< \infty} \|
e^{\gamma t}  v(\cdot, t)\|_{H^s} < \infty \} .\]  
%The space $Y_s $ is
%Banach space equipped with the norm
%\[ \|v\|_{Y_s} :=  \sup_{0\leq t< \infty} \| e^{\gamma
%t}  v(\cdot, t)\|_s.\] 

For any $v\in Y_s$,
\[ \| \Gamma (v)\|_{Y_s} \leq C_s \| u_0 \|_{H^s}  +\frac{C_sc_{s,s}}{2\gamma } \|
v\|_{Y_s}^2 .\] 
Choose
\[ 
\delta =\frac{\gamma}{4C_s^2c_{s,s}}, \qquad 
R=2C_s \delta .\]
Then, if $\|u_0\|\leq \delta$, for any $v\in Y_s$ with $\|v\|_{Y_s}
\leq R$,
\[ \|\Gamma (v)\|_{Y_s}\leq C_s \delta + \frac{C_sc_{s,s}}{2\gamma}(2C_s\delta)R
 \leq R.\] 
Moreover, for any $v_1,v_2\in Y_s$ with $\|v_1\|_{Y_s}\leq R$ and $\|v_2\|_{Y_s} \leq R$,
\[ \| \Gamma (v_1)-\Gamma (v_2)\|_{Y_s} \leq \frac12 \| v_1-v_2
\|_{Y_s} .\] 
The map $\Gamma$ is a contraction whose fixed point
$u\in Y_s $ is the desired solution satisfying
\[ \| u(\cdot, t)\|_{H^s} \leq 2C_s e^{-\gamma t} \| u_0\|_{H^s} \]
for any $t\geq 0$.
\end{proof}

Now we turn to the issue of the global stability of the system
\eqref{G8}-\eqref{G9}.  First we show that the system
\eqref{G8}-\eqref{G9} is globally exponentially stable in the space
$H^1(\T)$.
\begin{theorem}
\label{thm10} Let $a\in C^\infty (\T  )$ with $a\ne 0$, and let $\gamma >0$ be as in Lemma \ref{linear}.  Then for
any $R_0>0$, there exists a constant $C^*>0$ such that for any $u_0\in
H^1(\T)$ with $||u_0||_{H^1}\le R_0$, the corresponding solution $u$ of
\eqref{G8}-\eqref{G9} satisfies 
\be ||u(\cdot,t)||_{H^1} \le C^*e^{-\gamma t} ||u_0||_{H^1}\qquad \text{ for all } t\ge 0. \label{G10} \ee
\end{theorem}
Theorem \ref{thm10} is a direct consequence of the
following observability inequality.
\begin{proposition}\label{prop10}
Let $R_0>0$ be given. Then there exist two positive numbers  $T$ and
$\beta $ such that for any $u_0\in H^1(\T )$  satisfying
\be
\label{G11} ||u_0||_{H^1} \le R_0, \ee
 the corresponding solution $u$ of
\eqref{G8}-\eqref{G9} satisfies 
\be \label{G12} ||u_0||^2_{H^1} \le \beta \int_0^T ||a u (t) ||^2_{H^1} dt. \ee
\end{proposition}
Indeed, if \eqref{G12} holds, then it follows from the energy identity 
\be ||u(t)||^2_{H^1} =||u_0||^2_{H^1} 
-2 \int_0^t ||au (\tau ) ||^2_{H^1} d\tau \qquad \forall t \ge 0 \label{G13} \ee
that
\[
||u(T)||^2_{H^1} \le (1- 2\beta ^{-1}) ||u_0||^2_{H^1}.
\]
Thus
\[
||u(mT)||^2_{H^1} \le (1-2 \beta ^{-1})^{m} ||u_0||^2_{H^1}
\]
which gives by the semigroup property
\be ||u(t)||_{H^1} \le Ce^{-\kappa t} ||u_0||_{H^1}\qquad \text{ for all } t\ge 0, \ee
for some positive constants $C=C(R_0)$, $\kappa =\kappa (R_0)$. 
%We obtain a constant $\kappa$ independent of $R_0$ by noticing that for
%$t>\kappa ^{-1} (R_0)  \log [ C(R_0) ||u_0||_{H^1} ] $, the $H^1-$norm
%of $u(.,t)$ is smaller than $1$, so that we can take the $\kappa$
%corresponding to $R_0=1$.

Finally, we can replace $\kappa $ by the $\gamma $ given in Lemma
\ref{linear}. Indeed, let $t'=\kappa ^{-1} \log [1+CR_0\delta ^{-1}]$,  where $\delta $ is as given in
Proposition \ref{locale}. Then
for $\| u_0\| _{H^1}\le R_0$, $\|u( t')\|_{H^1} < \delta$, hence  for all $t\ge t'$ 
\[ \|u (t)\|_{H^1} 
\leq C'_1\|u(t')\|_{H^1} e^{-\gamma (t-t')}
\leq (C'_1 \delta/R_0) \| u_0 \|_{H^1} e^{-\gamma (t-t')}
\leq C^* e^{-\gamma t} \| u_0\| _{H^1}
\] 
where $C^*= (C'_1\delta / R_0) e^{\gamma t'}$.
\qed

\medskip
Now we present a proof of Proposition \ref{prop10}. Pick for the
moment any $T>2\pi / |c|$ (its value will be specified later on). We prove
the estimate \eqref{G12} by contradiction. If \eqref{G12} is not
true, then for any $n\ge 1$ \eqref{G8}-\eqref{G9} admits a solution
$u_n\in C([0,T];H^1(\T ))$ satisfying 
\be \label{G14} ||u_n(0) ||_{H^1}
\le R_0 \ee and \be \int_0^T ||au_n(t)||^2_{H^1} dt <\frac{1}{n}
||u_{0,n}||^2_{H^1} \label{G15} 
\ee 
where $u_{0,n}=u_n(0)$. Since
$\alpha _n := ||u_{0,n}||_{H^1} \le R_0$, one can choose a subsequence
of $(\alpha _n)$, still denoted by $(\alpha _n)$, such that
$\lim_{n\to \infty}\alpha _n = \alpha$. Note that $\alpha _n>0$ for
all $n$, by \eqref{G15}. Set $v_n=u_n/\alpha _n$ for all $n\ge 1$.
Then 
\be v_{n,t} - v_{n, txx} -c v_{n,xxx} + (c+1)v_{n,x} + \alpha _n v_n v_{n,x} =-a(1-\partial _x^2)[av_n] \label{G16} 
\ee 
and
\be
\label{G17} \int_0^T ||a v_n ||^2_{H^1} dt <\frac{1}{n}\cdot 
\ee
Because of \be ||v_n(0) ||_{H^1} =1, \label{G18} \ee the sequence
$(v_n)$ is bounded in $L^\infty (0,T; H^1(\T ))$, while $(v_{n,t})$
is bounded in $L^\infty (0,T;L^2(\T))$. From Aubin-Lions' lemma and
a diagonal process, we infer that we can extract a subsequence of $(v_n)$, still
denoted $(v_n)$, such that 
\ba
v_n\to v \quad &&\text{ in } C([0,T]; H^s(\T )) \quad \forall s<1 \label{G19}\\
v_n\to v \quad  &&\text{ in } L^\infty(0,T;H^1(\T ))\quad \text{weak}* \label{G20} 
\ea 
for some $v\in L^\infty (0,T;H^1(\T ))\cap C([0,T];H^s(\T ))$ for all $s<1$, Note that, by
\eqref{G19}-\eqref{G20}, we have that 
\be \alpha _n v_n\, v_{n,x} \to
\alpha v v_x \qquad \text{ in } L^\infty (0,T; L^2(\T ))\ \text{weak}*. 
\ee 
Furthermore, by \eqref{G17}, \be \int_0^T ||av ||^2_{H^1} dt
\le \liminf_{n\to\infty} \int_0^T ||av_n||^2_{H^1} dt =0. \label{D20bis}
\ee 
Thus, $v$ solves 
\ba
&&v_t -v_{txx} -cv_{xxx} + (c+1)v_x  +  \alpha vv_x = 0 \quad \text{ on } \T \times (0,T),\\
&&v=0 \quad \text{ on } \omega \times (0,T).
\ea
where $\omega $ is given in \eqref{D3}.
According to Corollary \ref{cor20}, $v\equiv 0$ on $\T \times (0,T)$.

We claim that $(v_n)$ is {\em linearizable} in the sense of  \cite{DGL}; that is,
if $(w_n)$ denotes the sequence of solutions to the linear
KdV-BBM equation  with the same initial data
\ba
&&w_{n,t} - w_{n,txx} -cw_{n,xxx} + (c+1) w_{n,x} =-a(1-\partial _x^2)[aw_n], \label{G21}\\
&&w_n(x,0) =v_n(x,0), \label{G22} 
\ea 
then 
\be \sup_{0\le t\le T} ||v_n(t) - w_n(t)||_{H^1} \to 0 \quad \text{ as } n\to \infty. \label{G23} \ee 
Indeed, if $d_n=v_n-w_n$, then $d_n$ solves
\begin{eqnarray*}
&&d_{n,t} -d_{n,txx} -cd_{n,xxx} +(c+1)d_{n,x} = -a(1-\partial _x^2)[ad_n] -\alpha _n v_n v_{n,x},\\
&&d_n(0)=0.
\end{eqnarray*}
Since $||W_a(t)||_{ {\mathcal L} (H^1(\T )) }\le 1$, we have from Duhamel formula that for $t\in [0,T]$
\[
||d_n(t)|| _{H^1} \le \int_0^T ||(1-\partial _x ^2)^{-1} (\alpha _n
v_nv_{n,x}) (\tau )||_{H^1} d\tau .
\]
Combined to \eqref{G19} and to the fact that $v\equiv 0$, this gives
\eqref{G23}. By Lemma \ref{linear}, we have that
\be \label{G24}
|| w_n(t) ||_{H^1} \le C_1 e^{- \gamma t} ||w_n(0)||_{H^1} \qquad \text{for all } t\ge 0. \ee 
From
\eqref{G24} and the energy identity for \eqref{G21}-\eqref{G22},
namely 
\be ||w_n(t)||^2_{H^1} - ||w_n(0)||^2_{H^1} 
=-2 \int_0^t ||aw_n(\tau )||^2_{H^1} d\tau, \ee 
we have for $Ce^{- \lambda T} <1$ 
\be
||w_n(0)||^2_{H^1} \le 2 (1- C_1^2 e^{-2 \gamma  T})^{-1} \int_0^T ||a
w_n(\tau )||^2_{H^1} d\tau. \label{G25} 
\ee 
Combined to \eqref{G17} and
\eqref{G23}, this yields $||v_n(0)||_{H^1}=||w_n(0)||_{H^1}\to 0$, which
contradicts \eqref{G18}. This completes the proof of Proposition
\ref{prop10} and of Theorem \ref{thm10}.\qed

\medskip
Next we show that the system \eqref{G8}-\eqref{G10} is exponentially stable in the space
$H^s (\T)$ for any $s\geq 1$.
\begin{theorem}
Let $a\in C^\infty (\T  )$ with $a\ne 0$  and $\gamma
>0$ be as given in Lemma \ref{linear}. For  any given $s\geq 1$
and $R_0>0$, there exists a constant $C>0$ depending only on $s$ and
$R_0$  such that for any $u_0\in H^s(\T)$ with $||u_0||_{H^s}\le R_0$,
the corresponding solution $u$ of \eqref{G8}-\eqref{G9} satisfies
\be 
\|u(\cdot ,t)\|_{H^s} \le Ce^{-\gamma t} \|u_0\|_{H^s}
\qquad \text{ for
all } t\ge 0.  \label{g-2} \ee
\end{theorem}
\begin{proof}
As before, rewrite the system in its integral form
\[ u(t)=W_a(t)u_0 -\frac12 \int ^t_0 W_a(t-\tau )( 1 -\partial
_x^2 )^{-1} (uu_x) (\tau) d\tau .\]  
For $u_0\in H^s (\T)$
with $\|u_0\|_{H^s} \leq R_0$, applying  Lemma \ref{linear}, 
Lemma \ref{bilinear} and Theorem \ref{thm10} yields that, for any 
$1\leq s\leq 1+\frac{1}{10}$,
\begin{eqnarray*}
 \| u(\cdot, t)\|_{H^s} 
 &\leq & C_se^{-\gamma t} \|u_0\|_{H^s} 
 + \frac{C_s c_{1,1,s}}{2}  \int^t_0 e^{-\gamma (t-\tau) }\| u(\cdot, \tau)\|_{H^1}^2 d\tau\\ 
 &\leq & C_se^{-\gamma t} \|u_0\|_{H^s} + \frac{C_s c_{1,1,s} (C^*)^2 }{2} \int ^t_0 e^{-\gamma (t-\tau) }
 e^{-2\gamma \tau }\| u_0\|_{H^1}^2 d\tau \\ 
&\leq & \left (C_s+\frac{C_s c_{1,1,s} (C^*)^2}{2\gamma }  \|u_0\|_{H^1}  \right ) e^{-\gamma t} \|u_0 \|_{H^s}
\end{eqnarray*}
for any $t\geq 0$. Thus the estimate (\ref{g-2}) holds for $1\leq s\leq
m_1:= 1+\frac{1}{10}$.  Similarly, for $m_1\leq s\leq m_2:=
1+\frac{2}{10}$, we have for $||u_0||_{H^s}\le R_0$
\begin{eqnarray*}
 \| u(\cdot, t)\|_{H^s} 
 &\leq & C_se^{-\gamma t} \|u_0\|_{H^s} + \frac{C_sc_{m_1,m_1,s}}{2} 
 \int^t_0 e^{-\gamma (t-\tau) }\| u(\cdot, \tau)\|_{H^{m_1}}^2 d\tau\\ 
&\leq & C_se^{-\gamma t} \|u_0\|_{H^s} + C(s,m_1,R_0) \int ^t_0 e^{-\gamma
(t-\tau) }e^{-2\gamma \tau }\| u_0\|_{H^{m_1}}^2 d\tau \\ 
&\leq & \left( C_s+C(s,m_1,R_0) ||u_0|| _{H^{m_1}} \gamma ^{-1} \right ) e^{-\gamma t}
\|u_0 \|_{H^s}.
\end{eqnarray*}
Thus the estimate (\ref{g-2}) holds for $1\leq s\leq m_2:=
1+\frac{2}{10}$. Repeating this argument  yields that the estimate
(\ref{g-2}) holds for $1\leq s\leq m_k:=1+\frac{k}{10}$ for
$k=1,2,\ldots$
\end{proof}

\section*{Acknowledgements}
The authors would like to thank E. Zuazua for having brought to
their attention the reference \cite{AT}. LR was  partially supported
by the Agence Nationale de la Recherche, Project CISIFS, grant
ANR-09-BLAN-0213-02.  BZ  was partially supported by a grant from
the Simons Foundation (\#201615 to Bingyu Zhang)


\begin{thebibliography}{9}
\addcontentsline{toc}{section}{References}
%\bibitem{BM} A. Beurling, P.  Malliavin,  {\em On Fourier transforms of measures with compact support}, Acta Math. {\bf 107} (1962) 291--309.
\bibitem{AT} P. Albano, D. Tataru, {\em Carleman estimates and boundary observability for a coupled parabolic-hyperbolic system},
Electron. J. Differential Equations {\bf 2000} (2000), no. 22, 1--15.
\bibitem{BS} J. M. Ball, M.  Slemrod, {\em Nonharmonic Fourier series and the stabilization of distributed semilinear control systems},
Comm. Pure Appl. Math. {\bf 32} (1979), no. 4, 555--587.
\bibitem{BBM} T. B. Benjamin, J. L. Bona, and J. J. Mahony, {\em Model equations for long waves in nonlinear dispersive systems}, Phil. Trans.
Royal Soc. London {\bf 272} (1972) 47--78.
\bibitem{BT} J. L.  Bona, N. Tzvetkov, {\em Sharp well-posedness for the BBM equation}, Discrete Contin. Dyn. Syst. {\bf 23} (2009), no. 4, 1241--1252.
\bibitem{Bourgain} J. Bourgain, {\em On the compactness of the support of solutions of dispersive equations}, Internat.
Math. Res. Notices {\bf 1997}, no. 9, 437--447.
%\bibitem{benjamin} T. B. Benjamin, {\em Internal waves of permanent form in fluids of great depth},
%J. Fluid Mech. {\bf 29} (1967), 559--592.
%\bibitem{case}  K. M. Case, {\em Benjamin-Ono related equations and their solutions}, Proc. Nat. Acad. Sci.
%U.S.A. {\bf 53}  (1965), 1092-1099.
%\bibitem{DR} K. D. Danov and M. S.   Ruderman, {\em Nonlinear waves on shallow water in the presence of a horizontal magnetic
%field}, Fluid Dynam. {\bf 18} (1983), 751--756.
\bibitem{Castro} C. Castro, {\em Exact controllability of the 1-d wave equation from a moving interior point}, preprint.
\bibitem{CZ1} C. Castro, E. Zuazua, {\em Unique continuation and control for the heat equation from a lower dimensional manifold}, SIAM J. Cont. Optim.,
{\bf 42} (4), (2005) 1400--1434.
\bibitem{Constantin} A. Constantin, {\em Finite propagation speed for the Camassa-Holm equation}, J. Math. Phys., 46 (2005), 023506, 4.
\bibitem{DPM}  M. Davila, G. Perla Menzala, {\em Unique continuation for the Benjamin-Bona-Mahony and
Boussinesq's equations}, NoDEA Nonlinear Differential Equations Appl. {\bf 5} (1998), no. 3, 367--382.
\bibitem{DGL} B. Dehman, P. G\'erard, G. Lebeau, {\em Stabilization and control for the nonlinear Schr\"odinger equation on a compact surface}, Math. Z. {\bf 254}
(2006), 729--749.
\bibitem{EGGP} S. Ervedoza, O. Glass, S. Guerrero, J.-P. Puel, {\em Local exact controllability for the 1-D compressible Navier-Stokes equation}, preprint.
\bibitem{EKPV} L. Escauriaza, C. E. Kenig, G. Ponce, L. Vega,  {\em  On uniqueness properties of solutions of the k-generalized KdV equations},
J. Funct. Anal. {\bf 244} (2007), no. 2, 504--535.
\bibitem{Glass} O. Glass, {\em Controllability and asymptotic stabilization of the Camassa-Holm equation}, J. Differential Equations
{\bf 245} (2008), no. 6, 1584-1615.
\bibitem{HESG} S. Hamdi, W. H. Enright, W. E. Schiesser, J. J. Gottlieb, {\em Exact solutions and invariants of motion for general
types of regularized long  wave equations}, Math. Comput. Simulation {\bf 65} (2004), no. 4-5, 535--545.
\bibitem{Hochschild} G. Hochschild, {\em The structure of Lie groups}, Holden-Day, Inc., San Francisco-London-Amsterdam 1965.
\bibitem{Hormander} L. H\"ormander, The analysis of linear partial differential operators. I. Distribution theory and Fourier analysis. Springer-Verlag, Berlin, 1990.
%\bibitem{CazenH1998}
%Thierry Cazenave and Alain Haraux, \emph{An introduction to semilinear
%  evolution equations}, Oxford Lecture Series in Mathematics and its
%  Applications, vol.~13, The Clarendon Press Oxford University Press, 1998.
%\bibitem{FR} H. O. Fattorini, D. L. Russell, {\em Exact controllability theorems for linear parabolic equations in one space dimension},
%Arch. Ration. Mech. Anal. {\bf 43} (1971) 272--292.
%\bibitem{Glass} O. Glass, {\em A complex-analytic approach to the problem of uniform controllability of a transport equation
%in the vanishing viscosity limit}, Journal of Functional Analysis {\bf 258} (2010), 852--868.
%\bibitem{Hardy} G.H. Hardy, E.M. Wright, {\em An Introduction to the Theory of Numbers}, Oxford University Press, Sixth edition, 2008.
%\bibitem{iorio2003} R. J. Iorio, {\em Unique continuation principles for the Benjamin-Ono equation},
%Differential and Integral Equation {\bf 16}, No. 11 (2003), 1281--1291.
%\bibitem{ishimori} Y. Ishimori, {\em Solitons in a one-dimensional Lennard/Mhy Jones lattice},
%Prog. Theoret. Phys. {\bf 68}, No. 2 (1982), 402--410.
\bibitem{Khapalov1} A. Khapalov, {\em Controllability of the wave equation with moving point control},
Appl. Math. Optim. {\bf 31} (1995), no. 2,  155--175.
\bibitem{Khapalov2} A. Khapalov, {\em Mobile point controls versus locally distributed ones for the controllability of the semilinear parabolic equation}, SIAM J. Cont. Optim.,
{\bf 40} (1) (2001) 231--252.
%\bibitem{Koosis}  P. Koosis, The Logarithmic Integral, vol. I,II, Cambridge Stud. Adv. Math., vol. 12, Cambridge University Press, Cambridge, 1988;
%Cambridge Stud. Adv. Math., vol. 21, Cambridge University Press, Cambridge, 1992.
%\bibitem{LT} I. Lasiecka, R. Triggiani, {\em Exact null controllability of structurally damped and thermo-elastic parabolic models.}
%Atti. Accad. Naz Lincei Cl. Sci. Fis. Mat. Natur. Rend. Lincei (9), Mat. Appl. {\bf 9} (1998), 43-69.
\bibitem{LV} N. A. Larkin, M. P. Vishnevskii, {\em Dissipative initial boundary value problem for the BBM-equation}, Electron. J. Differential Equations
{\bf 2008} (2008), no. 149, 1--10.
\bibitem{Laurent1} C. Laurent,  {\em Global controllability and stabilization for the nonlinear Schr\"odinger equation on an interval}, ESAIM Control Optim. Calc. Var. {\bf 16} (2010), no. 2, 356--379.
\bibitem{Laurent2} C. Laurent,  {\em Global controllability and stabilization for the nonlinear Schr\"odinger equation on some compact manifold of dimension 3}, SIAM J. Math. Anal. {\bf 42} (2010), no. 2, 785--832.
\bibitem{LRZ} C. Laurent, L. Rosier, B.-Y. Zhang, {\em Control and
stabilization of the Korteweg-de Vries equation on a periodic domain},
Comm. Partial Differential Equations {\bf 35}  (2010), no. 4, 707--744.
\bibitem{Leugering} G. Leugering, {\em Optimal controllability in viscoelasticity of rate type}, Math. Methods
Appl. Sci. {\bf 8} (1986), no. 3, 368--386.
\bibitem{LS} G. Leugering, E. J. P. G. Schmidt, {\em Boundary control of a vibrating plate with internal damping}, Math. Methods Appl. Sci.,
{\bf 11} (1989),  no. 5,  573--586.
\bibitem{LO} F. Linares, J. Ortega, {\em On the controllability and stabilization of the linearized Benjamin-Ono equation}, ESAIM Control Optim. Calc. Var. {\bf 11} (2005), no. 2, 204--218.
\bibitem{LR} F. Linares, L. Rosier,  {\em Exact controllability and stabilizability of  the Benjamin-Ono equation}, in preparation.
\bibitem{Lions} J.-L. Lions, {\em Pointwise control for distributed systems}, in Control and estimation in distributed parameter systems,
edited by H. T. Banks, SIAM, 1992.
\bibitem{Liu} K. Liu, {\em Locally distributed control and damping for the conservative systems},
SIAM J. Control Optim.  {\bf 35} (1997), no. 5, 1574--1590.
%\bibitem{LO} F. Linares, J. H. Ortega,  {\em On the controllability
%  and  stabilization of the linearized Benjamin-Ono equation},
%ESAIM Control Optim. Calc. Var., {\bf 11} (2005), 204--218.
%\bibitem{Levin}  B. Ya. Levin, Lectures on Entire Functions, Translations of Mathematical Monographs, American Mathematical Society, Vol. 150, 1996.
%\bibitem{LK} W. A. J. Luxemburg, J. Korevaar, {\em Entire functions and M\"untz-Sz\'asz type approximation}, Transactions of the American Mathematical Society {\bf 157}
%(1971), 23--37.
\bibitem{MRR} P. Martin, L. Rosier, P. Rouchon, {\em Null controllability of the structurally damped wave equation with moving point control}, preprint.
\bibitem{Mammeri} Y. Mammeri, {\em Unique continuation property for the KP-BBM-II equation}, Differential Integral Equations {\bf 22} (2009), no. 3-4, 393--399.
\bibitem{Micu} S. Micu, {\em On the controllability of the linearized Benjamin-Bona-Mahony equation},
SIAM J. Control Optim. {\bf 39}   (2001), no. 6, 1677--1696.
\bibitem{MORZ} S. Micu, J. Ortega, L. Rosier, B.-Y. Zhang, {\em Control and stabilization of a family of Boussinesq systems},  Discrete Contin. Dyn. Syst. {\bf 24} (2009), no. 2,  273--313.
\bibitem{MMC} P. J. Morrison, J. D. Meiss, J. R. Carey, {\em Scattering of regularized-long-wave solitary waves},
Phys. D {\bf 11} (1984), no. 3,  324--336.
\bibitem{Olver} P. J. Olver, {\em Euler operators and conservation laws of the BBM equation}, Math. Proc. Cambridge Philos. Soc. {\bf 85} (1979), no. 1, 143--160.
%\bibitem{MORZ} S. Micu, J. H. Ortega, L. Rosier, and B.-Y. Zhang,
%{\em   Control and stabilization of a family of Boussinesq systems},
%Discrete and Continuous Dynamical Systems, {\bf 24} (2009),
%no. 2, 273--313.
%\bibitem{ono} H. Ono, {\em Algebraic solitary waves in stratified fluids}, J. Phys. Soc. Japan
%{\bf 39} (1975), 1082--1091.
%\bibitem{pazy} A. Pazy, Semigroups of Linear Operators and Applications to Partial Differential Equations,
%Applied Mathematical Sciences, vol. 44, Springer-Verlag, New York Inc.
\bibitem{Pazy1983}
A. Pazy, {\em Semigroups of linear operators and applications to partial
differential equations}, Applied Mathematical Sciences, vol.~44,
Springer-Verlag, 1983.
\bibitem{PR} A. Pazoto, L. Rosier, {\em Stabilization of a Boussinesq system of KdV-KdV type}, System $\&$ Control Letters {\bf 57} (2008), 595--601. 
%\bibitem{PSM} M. Pellicer, J. Sol\`a-Morales, {\em Analysis of a viscoelastic spring-mass model}, J. Math.
%Anal. Appl. {\bf 294} (2), (2004)  687--698.
\bibitem{Rosier97} L. Rosier, {\em Exact boundary controllability for the Korteweg-de Vries equation on a bounded domain}, ESAIM Control Optim. Calc. Var.
{\bf 2} (1997), 33--55.
\bibitem{Rosiersurvey} L. Rosier, {\em A survey of controllability and stabilization results for partial
differential equations}, Revue des Syst\`emes, s\'erie Journal Europ\'een des Syst\`emes Automatis\'es,
Vol. {\bf 41} (2007), no. 3-4, 365--411.
\bibitem{RR} L. Rosier, P. Rouchon, {\em On the controllability of a wave equation with structural damping},
Int. J. Tomogr. Stat. {\bf 5} (2007), no. W07, 79--84.
%\bibitem{RZpreprint} L. Rosier, B.-Y. Zhang, {\em Unique continuation property and control for equations of Benjamin-Bona-Mahony type},
%preprint.
\bibitem{RZ2006} L. Rosier, B.-Y. Zhang, {\em Global stabilization
of the generalized Korteweg-de Vries equation posed on a finite domain}, SIAM J.
Control Optim. {\bf 45} (2006), 927--956.
\bibitem{RZ2007b}  L. Rosier, B.-Y. Zhang, {\em Local exact
controllability and stabilizability of the nonlinear  Schr\"odinger
equation on a bounded interval}, SIAM J. Control Optim. {\bf
48}\, (2009), no. 2, 972--992.
%\bibitem{RZ2008} L. Rosier and B.-Y. Zhang, {\em Exact boundary
%controllability of the nonlinear Schr\"odinger equation}, J.
%Differential Equations {\bf 246} (2009), 4129--4153.
%\bibitem{RZ2009bis} L. Rosier and B.-Y. Zhang, {\em Control and
%  stabilization of the Korteweg-de Vries equation: recent progresses},
%Jrl Syst $\&$ Complexity (2009) 22: 647--682.
\bibitem{RZ2009} L. Rosier,  B.-Y. Zhang, {\em Control and stabilization
of the nonlinear Schr\"odinger equation on rectangles},
M3AS: Math. Models Methods Appl. Sci. {\bf 20} (12) (2010), 2293--2347.
%\bibitem{rudin} W. Rudin, Real and Complex Analysis,
%McGraw-Hill, Inc, 1966.
\bibitem{Roumegoux} D. Roum\'egoux, {\em A symplectic non-squeezing theorem for BBM equation}, Dyn. Partial Differ. Equ. {\bf 7} (2010), no. 4, 289--305.
\bibitem{Russell} D. L. Russell, {\em Mathematical models for the elastic beam and their control-theoretic
implications}, in H. Brezis, M. G. Crandall and F. Kapper (eds), {\em Semigroup Theory and Applications}, Longman,
New York (1985).
\bibitem{RussellZhang96} D. Russell, B.-Y. Zhang, {\em Exact controllability and stabilizability of the Korteweg-deVries equation}, Trans. Amer. Math. Soc. {\bf 348} (1996) 3643--3672.
\bibitem{SS} J.-C. Saut, B. Scheurer, {\em Unique continuation for some evolution equations}, J. Differential Equations {\bf 66} (1987), no. 1, 118--139.
%\bibitem{Young} R. M. Young, {\em An Introduction to Nonharmonic Fourier Series}, Academic Press, 1980.
\bibitem{Yamamoto} M. Yamamoto, {\em One unique continuation for a linearized Benjamin-Bona-Mahony equation},
J. Inverse Ill-Posed Probl. {\bf 11} (2003), no. 5, 537--543.
\bibitem{Zhang92} B.-Y. Zhang, {\em Unique continuation for the Korteweg-de Vries equation}, SIAM J. Math. Anal. {\bf 23} (1992), no. 1, 55-71.
\bibitem{ZZ} X. Zhang, E. Zuazua, {\em Unique continuation for the linearized Benjamin-Bona-Mahony equation with
space-dependent potential}, Math. Ann. {\bf 325} (2003), no. 3, 543--582.
\end{thebibliography}
\end{document}